\documentclass[a4paper, 11pt, reqno]{amsart}
\usepackage{amsmath}
\usepackage{amsfonts}
\usepackage{amsthm}
\usepackage[utf8]{inputenc}
\usepackage{latexsym, amssymb}
\usepackage{epsfig}
\usepackage[pdftex]{hyperref}
\usepackage{indentfirst}
\usepackage{verbatim}
\usepackage[titletoc]{appendix}
\usepackage{extarrows}
\usepackage{url}
\usepackage{tikz}
\usetikzlibrary{matrix}

\newtheorem{thm}{Theorem}
\newtheorem*{thmI}{Min-max Theorem}

\newtheorem{lem}{Lemma}
\newtheorem{rem}{Remark}
\newtheorem{prop}{Proposition}

\newtheorem*{defin}{Definition}

\newcommand{\ga}{\gamma}

\newcommand{\pa}{\partial}
\newcommand{\re}{\mathbb{R}}
\newcommand{\N}{\mathbb{N}}
\newcommand{\si}{\Sigma}
\DeclareMathOperator{\length}{length}
\DeclareMathOperator{\area}{area}
\DeclareMathOperator{\vol}{vol}
\DeclareMathOperator{\sing}{sing}
\DeclareMathOperator{\ind}{ind}
\DeclareMathOperator{\nul}{nul}
\DeclareMathOperator{\spt}{spt}
\DeclareMathOperator{\inj}{inj}
\DeclareMathOperator{\im}{im}
\DeclareMathOperator{\dmn}{dmn}
\DeclareMathOperator{\Tan}{Tan}

\newenvironment{dem1}[1][\noindent \textit{Proof of Theorem \ref{thm:width.charac}. }]{#1}{\hfill$\square$ \vspace{0.3cm}}

\newenvironment{dem3}[1][\noindent \textit{Proof of Theorem \ref{thm:non_con}. }]{#1}{\hfill$\square$ \vspace{0.3cm}}

\newenvironment{dem4}[1][\noindent \textit{Proof of Theorem \ref{thm:Bolza}. }]{#1}{\hfill$\square$ \vspace{0.3cm}}

\newcommand{\mres}{\mathbin{\vrule height 1.6ex depth 0pt width
0.13ex\vrule height 0.13ex depth 0pt width 1.3ex}}

\title[On the first width of hyperbolic surfaces]{On the first width of hyperbolic surfaces: multiplicity and lower bounds}
\author{Vanderson Lima}
\address{Universidade Federal do Rio Grande do Sul\\
  Instituto de Matem\'atica e Estat\'istica\\
 Porto Alegre, RS - 91509-900, Brazil}
\email{vanderson.lima@ufrgs.br}

\begin{document}

\begin{abstract}
On a closed Riemannian surface of negative curvature, we prove a characterization for configurations of closed geodesics arising from one parameter Allen-Cahn min-max constructions. One of the facts we conclude is that every geodesic occurs with multiplicity one. As an application we obtain a uniform sharp lower bound for the first min-max width of closed hyperbolic surfaces and prove it is only attained asymptotically. Moreover, we compute the first width of the Bolza surface and of some hyperbolic surfaces with small systoles. 
\end{abstract}

\maketitle

\section{Introduction}

The volume spectrum of a closed Riemannian manifold $(M^{n+1},g)$, defined by Gromov \cite{Gro}, Guth \cite{Gut} and Marques-Neves \cite{MaNe2}, is a sequence 
$$\big\{\omega_{k}(M^{n+1},g)\big\}_{k=1}^{\infty} \subset (0,+\infty),$$ 
which can be thought as a nonlinear analogue of the Laplace sprectrum of $(M^{n+1},g)$. For instance, $\{\omega_{k}\}$ satisfy a Weyl Law, as conjectured by Gromov and proved by Liokumovich-Marques-Neves \cite{LMN}. The number $\omega_{k}(M^{n+1},g)$ is called the \emph{$k$-width} of $(M^{n+1},g)$ and can be characterized by the following.

\begin{thmI}\label{thm:Min-max}
Let $(M^{n+1},g)$ be a closed Riemannian manifold.
\begin{itemize}
\item[(a)] (Marques-Neves \cite{MaNe1}) If $3 \leq n+1 \leq 7$, then
$$\omega_{k}(M^{n+1},g) = \sum_{i = 1}^{\ell_k}m_{k}^{i}\vol_{g}(\si_{k}^{i}),$$
where $m_{k}^{i} \in \mathbb{N}$ and $\{\si_{k}^{i}\}_{1\leq i\leq \ell_k}$ is a collection of disjoint, connected, embedded closed minimal hypersurfaces of $(M^{n+1},g)$.\\

\item[(b)] (Chodosh-Mantoulidis \cite{CM2})
If $n+1 = 2$, then
$$\omega_{k}(M^{2},g) = \sum_{i = 1}^{\ell_k}m_{k}^{i}\length_{g}(\sigma_{k}^{i}),$$
where $m_{k}^{i} \in \mathbb{N}$ and $\{\sigma_{k}^{i}\}_{1\leq i\leq \ell_k}$ is a collection of primitive, immersed closed geodesics of $(M^{2},g)$ which are distinct (not necessarily disjoint).
\end{itemize}
\end{thmI}

Item (a) of the previous result is part of the so called \emph{Amgren-Pitts min-max theory}, first introduced by the aforemented authors \cite{Al2,P}, with further regularity results by Schoen-Simon \cite{ScSim}. Marques and Neves \cite{MaNe1,MaNe2,MaNe3} significantly developed the theory, as part of their program to develop a Morse Theory for the volume functional, which led to very important applications in geometry in the last years, see \cite{MaNe,AMN,IMN,MNS,Song}. 

In dimension $2$, Almgren-Pitts min-max only garantees the existence of a \emph{geodesic network} \cite{P2}, so in order to obtain item (b), Chodosh-Mantoulidis used the \emph{phase transitions} variant of min-max. This is motivated by the connection between solutions of the Allen-Cahn equation and minimal hypersurfaces, explored for instance in \cite{Mod,St,PT,HT,TW}. The connection of the Allen-Cahn functional with the Almgren-Pitts min-max theory was first studied by Guaraco \cite{G} and Gaspar-Guaraco \cite{GG}, and then by Mantoulidis (for $n+1=2$). In particular, Gaspar-Guaraco \cite{GG} introduced a sequence of $k$-widths in the phase-transitions setting and compared it to the volume spectrum. The equality of the two sequences was then proved by Dey \cite{Dey}.

The numbers $m^{i}_{k}$ appearing in the Min-max Theorem are called the \emph{multiplicities} associated with the components $\si_{k}^{i}$ (or $\sigma_{k}^{i}$). The \emph{Multiplicity One Conjecture} of Marques-Neves states that for $3 \leq n+1 \leq 7$ and a generic metric $g$, there is a collection $\{\si_{k}^{i}\}_{1\leq i\leq \ell_k}$, such that $m^{i}_{k} = 1$, $\forall\,k$ and $\forall\, i \in [1,\ell_k]$. This was proved first by Marques-Neves \cite{MaNe1} for $k=1$ and later by Chodosh-Mantoulidis \cite{CM1} for all $k$ and $n=3$ (in the context of the Allen-Cahn min-max theory). The full conjecture was then settled by Zhou \cite{Zho} (see also \cite{CL}).

However, in dimension $2$ the multiplicity one property does not hold. This was first observed by Aiex \cite{Aie} in some ellipsoids. Very recently Marx-Kuo, Sarnataro and Stryker \cite{MKSS} proved the stronger result that there are open sets $U_k$ of metrics of positive curvature on $\mathbb{S}^2$, such that multiplicity one property does not hold for $\omega_{\ell}(\mathbb{S}^2,g)$, $\ell=1,\ldots,k$, where $g \in U_k$.

In this paper we study the first min-max width $\omega_1$ of surfaces of negative curvature. In particular, we are concerned with the question of multiplicity. Our first result is the following.

\begin{thm}\label{thm:width.charac}
Let $(M,g)$ be a closed oriented Riemannian surface of genus $m \geq 2$, such that $K_{g} < 0$. Then, there exist a primitive figure eight geodesic $\sigma_{\infty}$ and a possibly empty collection of primitive simple closed geodesics $\sigma_{1},\dots,\sigma_{N}$, $N \leq 3m-3$, which are disjoint from $\sigma_{\infty}$ and pairwise disjoint, such that
\begin{equation}\label{eq:width.mult}
\omega_1(M,g) = \length(\sigma_{\infty}) + \sum_{j=1}^N\length(\sigma_j).
\end{equation}
Moreover,
\begin{equation}\label{eq:width.ineq}
\omega_1(M,g) > 2\big(-\inf_{M}K_g\big)^{-1/2}\mathrm{arccosh}\,3.
\end{equation}
\end{thm}

The equation \eqref{eq:width.mult} express the fact $\omega_1$ is realized with multiplicity one. Concerning the bound \eqref{eq:width.mult}, in the case the surface is hyperbolic, we prove that the bound is sharp and is realized asymptotically. More precisely, the following holds.

\begin{thm}\label{thm:sharp}
There exists $a_0 > 0$ such that for all $a \in (0,a_0]$ and for every integer $m \geq 2$, there is a closed hyperbolic surface $S_{m,a}$ of genus $m$ and systole equal to $a$, such that:
\begin{align*}
\omega_1(S_{m,a}) = \length(\sigma_{a}) \xlongrightarrow{a \to 0^{+}} 2\,\mathrm{arccosh}\,3,
\end{align*}
where $\sigma_a$ is a primitive figure eight geodesic of $S_{m,a}$.
\end{thm}

Recall the following facts concerning the Laplace spectrum $\{\lambda_{k}\}_{k=1}^{\infty}$ of hyperbolic surfaces:
\begin{itemize}
\item (Buser, 1977 \cite{Bus0}) There exists $\epsilon_0 > 0$ such that for all $\epsilon \in (0,\epsilon_0]$ and for every integer $m \geq 2$, there exist a closed hyperbolic surface $S_{m,\epsilon}$ of genus $m$ such that
$$\lambda_{2m-2}(S_{m,\epsilon}) \leq \frac{1}{4} + \epsilon.$$
\item (Otal-Rosas, 2009 \cite{OR}) Let $S$ be a closed hyperbolic surface of genus $m \geq 2$. Then
$$\lambda_{2m-2}(S) > \frac{1}{4}.$$
(The non-strict inequality was conjectured to hold by Buser and Schmutz \cite{Bus}.)
\end{itemize}
We can think of Theorems \ref{thm:width.charac} and \ref{thm:sharp} as analogues of these results for the case of the volume spectrum.

Theorem \ref{thm:width.charac} is a consequence of a general result about the convergence of level sets of the Allen-Cahn equation. We prove that in closed surfaces of negative curvature, configurations of closed geodesics arising from Allen-Cahn solutions whose energy is bounded, the morse index is at most one and the sum of Morse index and nullity is at least one, have multiplicity one, see Theorem \ref{thm:mult.one}. We remark that for general surfaces the multiplicity one property in the Allen-Cahn functional variational theory does not always hold, as proved recently by Liu-Pacard-Wei \cite{LPW}, however it holds if the limit interface is a strictly stable simple closed geodesic, which is a particular case of a result by Guaraco-Marques-Neves \cite{GMN}.

Observe that in Theorem \ref{thm:width.charac} one always has a figure eight geodesic present on the configuration realizing the width.  However is not clear if there is an example of a surface where the simple closed geodesics also appear. We prove this can happen.

\begin{thm}\label{thm:non_con}
There is $L_0 > 0$ such that, for all $L \in (0,L_0)$, there is a closed hyperbolic surface $S_L$ of genus $2$, whose systole is equal to $L$ and such that
$$\omega_1(S_L) = \length(\sigma_L) + \length(\gamma_L),$$
where $\sigma_L,\gamma_L$ are disjoint primitive closed geodesics, such that $\sigma$ is a figure eight curve and $\gamma_L$ is a systole (hence it is simple).
\end{thm}

The surfaces in the Theorems \ref{thm:sharp} and \ref{thm:non_con} have small systoles, and this property plays a crucial hole in the proofs of those results. One can wonder if is possible to compute the $1$-width of surfaces with large systole. Recall that the \emph{Bolza surface} $\mathrm{B}$ maximizes the systole among closed hyperbolic surfaces of genus $2$, \cite[Theorem Theorem 5.3]{SS}.

\begin{thm}\label{thm:Bolza}
The first width of the Bolza surface satisfies
$$\omega_1(\mathrm{B}) = \length(\sigma),$$
where $\sigma$ is a separating figure eight geodesic, which has the least length among all such curves in $\mathrm{B}$.
\end{thm}

\noindent
{\bf Plan of the paper}: In Section \ref{sec:2} we recall the results of min-max theory used along the text. Next, in Section \ref{sec:3}, we prove Theorem \ref{thm:width.charac}. In Section \ref{sec:4} we present applications to hyperbolic surfaces.\\  %Finally, in Section \ref{sec:5} we prove the inequality $\omega_1 < \omega_2$.\\

\noindent
{\bf Acknowledgements}: I wish to express my gratitude to professor André Neves for suggesting me this project and for very stimulating discussions about it. Also, I would like to thank Pedro Gaspar for his willingness to answer questions about the Allen-Cahn equation on several different occasions.  I also thank professor Fernando C. Marques for his interest in this work. Finally, I am grateful to the University of Chicago for the hospitality, where part of the research in this paper was developed. %Finally, I am grateful to Lucas Ambrozio for explaining me some of the technical details of \cite{AMN}.
During the preparation of this work I was supported by the  Simon Investigator Grant of Professor André Neves, by the Grant Serrapilheira/FAPERGS 23/2551-0001590-0 and by the CNPq-Brazil Grant 406666/2023-7.

\section{Min-max theory}\label{sec:2}

Let $I(1,j)$ be the cube complex on $I := [0,1]$ whose $1$-cells are 
$$[0,3^{-j}],[3^{-j},2 \cdot 3^{-j}],\dots,[1-3^{-j},1].$$
Define $I(m,j)$ as the cell complex on $I^m$ given by 
$$I(1,j)\otimes\cdots\otimes I(1,j).$$
A \emph{cubical subcomplex} $X\subset I^k$ is a subcomplex of $I(k,j)$ for some $j \in \N$. If $X$ is a subcomplex of $I(k,j)$, for $\ell\in\N$, denote by $X(\ell)$ the subcomplex of $I(k,j+\ell)$ given by the union of all cells whose support is contained in some cell of $X$. We denote by $X(\ell)_0$ the set of $0$-cells in $X(\ell)$.

Let $\big(M,g\big)$ be a closed oriented Riemannian $2$-manifold isometrically embedded in some Euclidean space. In the following we are going to suppress the reference to the metric $g$.
\vspace{0.2cm}

\noindent
{\bf Notation}:
\begin{itemize}
\item $\mathcal{V}_1(M)$ - the space of $1$-varifolds on $M$;
\item $\mathcal{I}\mathcal{V}_1(M)$ - the space of integral rectifiable $1$-varifolds on $M$.
\item $\mathbf{I}_k(M;\mathbb{Z}_2)$ - the set of $k$-dimensional mod $2$ flat chains in $\re^J$ with support in $M$;
\item $\mathcal{Z}_1(M;\mathbb{Z}_2)\subset \mathbf{I}_1(M;\mathbb{Z}_2)$ - the set of cycles, i.e., the set of $T \in \mathbf{I}_1(M;\mathbb{Z}_2)$ such that $T = \pa U$ for some $U \in \mathbf{I}_2(M;\mathbb{Z}_2)$.
\item If $\sigma$ is a primitive closed curve in $M$ we denote by $\mathbf{v}(\sigma,\mathbf{1}_{\sigma})$ the $1$-varifold of multiplicity one induced by $\sigma$.
\end{itemize}

For $T \in \mathbf{I}_1(M;\mathbb{Z}_2)$, denote by $|T|,\Vert T\Vert$ the associated integral varifold and Radon measure on $M$ respectively. Also, for $V \in \mathcal{V}_1(M)$, we denote by $\Vert V\Vert$ the associated Radon measure on $M$. Recall we can endow $\mathbf{I}_1(M;\mathbb{Z}_2)$ with three important topologies \cite[page 66]{P}, induced by the following objects:
\begin{itemize}
\item[(a)] the \emph{flat metric} $\mathcal{F}$;
\item[(b)] the $\mathbf{F}$-metric;
\item[(c)] the \emph{mass functional} $\mathbf{M}$.
\end{itemize}

We will use the notation $\mathcal{Z}_1(M,\mathcal{T};\mathbb{Z}_2)$ for the set $\mathcal{Z}_1(M;\mathbb{Z}_2)$ endowed with one the topologies $\mathcal{T} = \mathcal{F},\mathbf{F},\mathbf{M}$.

\subsection{The $k$-widths}\label{subsec:1-width}

In what follows $X$ denotes a cubical subcomplex of some $I^m$. %Every such cubical complex is homeomorphic to a finite simplicial complex and vice-versa (see \cite[Chapter 4]{BP}).

Almgren proved that $\mathcal{Z}_{1}(M;\mathbb{Z}_{2})$ is weakly homotopic to $\re P^{\infty}$ \cite{Al} (for a simpler proof see \cite[Theorem 5.1]{MaNe3}). So
$$H^{1}(\mathcal{Z}_{1}(M;\mathbb{Z}_{2});\mathbb{Z}_{2}) = \mathbb{Z}_{2}.$$
Denote by $\bar \lambda$ the generator.

\begin{defin}\label{defi:sweepout}
A map $\Phi : X\to \mathcal{Z}_{1}(M;\mathbb{Z}_{2})$ is a $k$-sweepout if it is continuous in the flat topology $\mathcal{F}$ and $\Phi^{*}(\bar\lambda^{k}) \neq 0$,
where 
$$\bar\lambda^{k} = \bar\lambda\smile\ldots\smile\bar\lambda.$$
Moreover, $\Phi$ is said to have no concentration of mass if
$$ \lim_{r \to 0} \sup \big\{\Vert\Phi(x) \Vert(B_r(p)) : x \in X, p \in M\big\} = 0. $$
\end{defin}

\begin{defin}\label{defi:p-width}
We define $\mathcal{P}_{k} = \mathcal{P}_{k}(M)$ to be the set of all $k$-sweepouts $\Phi: X \to \mathcal{Z}_{1}(M;\mathbb{Z}_{2})$ with no concentration of mass, where $X$ is any cubical subcomplex. The \emph{$k$-width} of $(M,g)$ is defined by
$$\omega_{k}(M,g) = \inf_{\Phi \in \mathcal{P}_{k}} \sup_{x \in \dmn(\Phi)}\mathbf{M}\big(\Phi(x)\big).$$
\end{defin}

Define 
$$\mathcal{P}^{\mathbf{F}}_{k,m} : = \{\Phi \in \mathcal{P}_{k}: \dmn(\Phi) \subset I^m\ \text{and}\ \Phi\ \text{is}\ \mathbf{F}\text{-continuous}\}.$$

\begin{lem}[Li \cite{Li}, Chodosh-Mantoulidis \cite{CM2}] \label{lem:bound.dim}
If $m = 2k+1$, then
$$\omega_{k}(M,g) = \inf_{\Phi \in \mathcal{P}^{\mathbf{F}}_{k,m}} \sup_{x \in \dmn(\Phi)}\mathbf{M}\big(\Phi(x)\big).$$
\end{lem}

\subsection{Almgren--Pitts theory}\label{subsec:cont-AP}

\begin{defin}\label{defi:ap.homotopy}
	Let $X$ be a cubical subcomplex and consider a continuous map $\Psi : X \to \mathcal{Z}_1(M; \mathbf{F}; \mathbb{Z}_2)$. We define the \textit{homotopy class} of $\Phi$ to be the set
	\begin{align*}
\Pi & := \{\Phi : X \to \mathcal{Z}_1(M; \mathbf{F}; \mathbb{Z}_2);\, \Phi\ \text{is continuous}\\ 
&\quad\quad \text{and it is homotopic to } \Psi \text{ in the } \mathcal{F} \text{-topology} \}.
\end{align*}
The Almgren-Pitts width of $\Pi$ is defined by
$$\mathbf{L}_{\textrm{AP}}(\Pi) = \inf_{\Phi \in \Pi} \sup_{x\in X} \mathbf{M}(\Phi(x)).$$
\end{defin}

We say $\{\Phi_i\}_{i=1}^{\infty}  \subset \Pi$ is a minimizing sequence if   
$$\limsup_{i\to\infty}\sup_{x\in X}\mathbf{M}(\Phi_i(x)) = \mathbf{L}_\textrm{AP}(\Pi).$$
The image set $\mathbf{\Lambda}(\{\Phi_i\})$ of $\{\Phi_i\}$ is defined as the set of varifolds $V \in \mathcal{V}_1(M)$, such that there is $i_\ell\to\infty$ and $x_\ell\in X$ with 
$$\lim_{\ell\to\infty}\mathbf{F}(|\Phi_{i_\ell}(x_\ell)|,V) = 0,$$ 
and the \emph{critical set} of $\{\Phi_i\}_{i=1}^{\infty}$ is defined as
$$\mathbf{C}(\{\Phi_i\}) = \{V\in\mathbf{\Lambda}(\{\Phi_i\}) : \Vert V\Vert(M) = \mathbf{L}_\textrm{AP}(\Pi)\}. $$

By the following result one can improve a minimizing sequence into a new minimizing sequence whose critical set consists only of stationary varifolds.

\begin{prop}\label{prop:pull.tight}
Following the previous notations, suppose $\mathbf{L}_{\textrm{AP}}(\Pi)>0$. Then, for any minimizing sequence $\{\Phi_i\}_{i=1}^\infty \subset \Pi$, there exists another minimizing sequence $\{\Phi_{i}^{\ast}\}_{i=1}^\infty \subset \Pi$, such that $\mathbf{C}\big(\{\Phi_{i}^{\ast}\}_{i=1}^\infty\big) \subset \mathbf{C}(\{\Phi_i\}_{i=1}^\infty)$ and every varifold in $\mathbf{C}\big(\{\Phi_{i}^{\ast}\}_{i=1}^\infty\big)$ is stationary. 
\end{prop}

\subsection{Phase transition theory} \label{subsec:AC} 

\begin{defin} \label{defi:ac.potential.general}
A smooth function $W : \re \to \re$ is said to be a double-well potential if it satisfies the following properties:
\begin{enumerate}
\item[(a)] $W \geq 0$;
\item[(b)] $W(-t) = W(t)$, for all $t \in \re$;
\item[(c)] $t W^{\prime}(t) < 0$, for $0 < |t| < 1$;
\item[(d)] $W^{\prime\prime}(\pm 1) > 0$.
\end{enumerate}
\end{defin}

Fix a double-well potential $W$. For $\varepsilon > 0$, define the \emph{$\varepsilon$-phase transition energy} of a function $u \in C^\infty(M)$ to be 
\begin{equation} \label{eq:ac.energy}
E_{\varepsilon}(u) = \int_{M} \left(\frac \varepsilon 2 |\nabla u|^{2} + \frac{1} {\varepsilon} W(u)\right)d\mathcal{H}^2.
\end{equation}
A function $u$ is a critical point of $E_\varepsilon$ if, and only if, solves the equation
\begin{equation} \label{eq:ACeqt}
\varepsilon^2 \Delta u = W'(u).
\end{equation}

The second variation of $E_\varepsilon$ at a critical point $u$ is given by
$$D^2 E_\varepsilon(u)\cdot(v,w) = \int_M \big(\varepsilon \langle\nabla v,\nabla w\rangle + \varepsilon^{-1}W''(u) vw\big)d\mathcal{H}^2,\quad \forall\, v, w \in C^\infty(M).$$
The \emph{Morse index} of $u$, denoted $\ind_{E_\varepsilon}(u)$, is a non negative integer defined as the maximum of the following set
$$\big\{\dim V; V\subset C^\infty_c(U) \textrm{ linear subspace with } D^2 E_\varepsilon(u)\cdot(v,v) < 0, \forall\, v\in V\setminus\{0\}\big\},$$
and the \emph{nullity} of $u$, denoted $\nul_{E_{\varepsilon}}(u)$, is defined as
$$\dim\big\{v \in W^{1,2}(M);\, D^2 E_\varepsilon(u)\cdot(v,v) = 0\big\}.$$

The min-max construction of solutions of \eqref{eq:ACeqt} was described by Gaspar and Guaraco \cite{GG,G} (see also \cite{Dey}). As in the previous subsection, $X$ is a cubical subcomplex of $I^{m}$ for some $m \in \N$. 
Fix any double cover $\pi : \widetilde X \to X$. Write $\Pi$ for the $\mathcal{F}$-homotopy class of $\mathbf{F}$-continuous maps corresponding to $\pi$, i.e., $\Phi : X\to \mathcal{Z}_{1}(M;\mathbf{F};\mathbb{Z}_{2})$ is in $\Pi$ whenever
$$
\ker(\Phi_{*} : \pi_1(X) \to \pi_1(\mathcal{Z}_1(M;\mathbb{Z}_2))) =\im \pi_{*} \subset \pi_{1}(X).
$$
Note that fixing $\Pi$ is the same as fixing the double cover $\pi:\widetilde X\to X$.

Since $\mathbf{I}_{2}(M;\mathbb{Z}_{2})$ is contractible and paracompact, 
$$\partial : \mathbf{I}_{2}(M;\mathbb{Z}_{2}) \to \mathcal{Z}_{1}(M;\mathbb{Z}_{2})$$
is a $\mathbb{Z}_{2}$-principal bundle. Since $W^{1,2}(M)\setminus\{0\}$ is contractible and paracompact with a free $\mathbb{Z}_{2}$ action $u\mapsto -u$, it is the total space of a $\mathbb{Z}_{2}$-principal bundle. We denote by $\widetilde \Pi$ the space of $\mathbb{Z}_{2}$-equivariant maps $h: \widetilde X \to W^{1,2}(M)\setminus\{0\}$. We define \emph{$\varepsilon$-phase transition width} of $\widetilde \Pi$ by 
$$\mathbf{L}_{\varepsilon}(\widetilde \Pi) = \inf_{h\in \widetilde\Pi} \sup_{x\in\widetilde X} E_{\varepsilon}\big(h(x)\big).$$
We say that $u \in W^{1,2}(M)\setminus\{0\}$ is a \emph{min-max critical point} if $E_{\varepsilon}(u) = \mathbf{L}_{\varepsilon}(\widetilde\Pi)$ and there is a minimizing sequence $\{h_{i}\}_{i=1}^{\infty}\subset \widetilde\Pi$ with
$$\lim_{i\to\infty} d_{W^{1,2}(M)}(u,h_{i}(\widetilde X)) = 0.$$

\begin{prop}[\cite{G,GG}]\label{prop:AC.min.max}
Suppose $\mathbf{L}_{\varepsilon}(\widetilde \Pi) < E_{\varepsilon}(0) = \area(M)/\varepsilon$. Then there is a min-max critical point $u_{\varepsilon}$ of $E_{\varepsilon}$. Moreover the function $u_\varepsilon$ is a solution of \eqref{eq:ACeqt}, satisfies the bound $|u_\varepsilon| < 1$, and has its index is such that $\ind_\varepsilon(u_\varepsilon) \leq \dim X = k$.
\end{prop}

To every solution $u$ of \eqref{eq:ACeqt} one can associate a 1-varifold on $M$, defined as the unique  $V_\varepsilon[u] \in \mathcal{V}_1(M)$ such that
$$V_\varepsilon[u](f) := h_0^{-1} \int_M \varepsilon |\nabla u|^2 f(x, \Tan_x \{ u = u(x) \})\,d\mathcal{H}^{2},
\quad \forall\, f \in C^0(G_1(M)).$$

\begin{prop}[{\cite[Theorem 1]{HT}, \cite{T}, \cite[Appendix B]{G}}]\label{prop:HT.theory}
Let $M$ be a surface and consider the following objects: 
\begin{itemize}
\item a open set $U \subset M$;
\item complete Riemannian metrics $\{g_i\}_{i=1}^\infty$, $g_\infty$ on $M$ with $\lim_i g_i = g_\infty$ in $C^\infty_\textnormal{loc}(M)$;
\item a sequence $\{\varepsilon\}_{i=1}^\infty\subset (0,\infty)$ with $\lim_i \varepsilon_i=0$;
\item a sequence $\{u_i\}_{i=1}^\infty \subset C^\infty_\textnormal{loc}(U)$, where $u_{i}$ is a critical point of $E_{\varepsilon_i}$, $\forall\, i$.
\end{itemize}
Suppose
$$\Vert u_i \Vert_{L^\infty(U)} \leq 1 \text{ and } (E_{\varepsilon_i}\mres(U,g_i))(u_i) \leq E_0, \forall\, i$$
Then, after passing to a subsequence, we have:
\begin{itemize}
\item[(a)] $\lim_i u_i = u_\infty$ in $L^1_\textnormal{loc}(U)$, $u_\infty \in BV_\textnormal{loc}(U)$, $u_\infty = \pm 1$ a.e.\ on $U$;
\item[(b)] $\lim_i V_{\varepsilon_i}[u_i] \mres G_1(U)  = V^\infty$ for a stationary integral $1$-varifold $V^\infty \in \mathcal{I}\mathcal{V}_1(U)$;
\item[(c)] $\lim_i(h_0^{-1}E_{\varepsilon_i}\mres (U',g_i))[u_i] = \Vert V^\infty\Vert(U')$ for all $U' \subset\subset U$,
\item[(d)] $\lim_i \{u_i=t\} \cap U' = \spt \Vert V^\infty\Vert \cap U'$ in the Hausdorff topolgy, for all $U' \subset\subset U$ and all $t \in (-1,1)$;
\item[(e)] the density of $V^\infty$ is a.e. odd on $\partial^*\{u_\infty = +1\}\cap U$ and a.e. even on $\spt \Vert V^\infty\Vert \cap U \setminus \partial^* \{u_\infty = +1\}$. 
\end{itemize}
\end{prop}

In the context of previous result, the set $\Gamma = \spt(V^\infty)$ is called a \emph{limit interface}.
Concerning the multiplicities and the Morse index of limit interfaces we have the following results.

\begin{thm}[Hiesmayr \cite{H}, Gaspar \cite{Ga}]\label{thm:ind.lb}
Let $\{u_{\epsilon_i}\}$ a sequence of solutions to \eqref{eq:ACeqt} with $\lim\epsilon_i = 0$. Suppose the following conditions are true:
\begin{itemize}
\item[(a)]there are positive constants $c_0$ and $E_0$, and a nonnegative integer $p$ such that
$$\limsup_i \sup_M|u_{\varepsilon_i}| \leq 1, \quad \limsup_i E_{\varepsilon_i}(u_{\epsilon_i}) \leq E_0, \quad  \limsup_i \ind_{E_{\varepsilon_i}}(u_{\varepsilon_i}) \leq p;$$
\item[(b)] a limit interface $\Gamma$ is a simple closed geodesic.
\end{itemize}
Then $\ind(\Gamma) \leq p$.
\end{thm}

\begin{thm}[Chodosh-Mantoulidis \cite{CM1}]\label{thm:ind.ub}
Let $\{u_{\epsilon_i}\}$ a sequence of solutions to \eqref{eq:ACeqt} with $\lim\epsilon_i = 0$. Suppose the following conditions are true:
\begin{itemize}
\item[(a)]there are positive constants $c_0$ and $E_0$, and a nonnegative integer $p$ such that
$$\limsup_i \sup_M|u_{\epsilon_i}| \leq 1, \quad \limsup_i E_{\epsilon_i}(u_{\epsilon_i}) \leq E_0;$$
\item[(b)] a limit interface $\Gamma$ is a simple closed geodesic and it occurs with multiplicity one.
\end{itemize}
Then for $\epsilon_i$ sufficiently small we have $$\ind(\Gamma) + \nul(\Gamma) \geq \ind_{E_{\varepsilon_i}}(u_{\epsilon_i}) + \nul_{E_{\varepsilon_i}}(u_{\epsilon_i}).$$
\end{thm}

%\begin{thm}[Guaraco-Marques-Neves \cite{GMN}]\label{thm:GMN} Let $U \subset M$ be an open set and $\Gamma \subset U$ be a simple closed geodesic which is the limit interface of a sequence of solutions to \eqref{eq:ACeqt} on $U$. If $\Gamma$ is strictly stable, then $\Gamma$ has multiplicity one as a limit interface. \end{thm}

\subsection{Comparison between the min-max theories}\label{subsec:dey}

Recall that the \emph{heteroclinic solution} $\mathbb{H}: \re \to (-1, 1)$, is the unique solution of \eqref{eq:ACeqt} on $\re$ with $\varepsilon=1$, such that
	\begin{equation*}
\mathbb{H}(0) = 0, \quad \lim_{t \to \pm \infty} \mathbb{H}(t) = \pm 1.
\end{equation*}
We denote $h_0 = \Vert\mathbb{H}\Vert_{L^2}^{2}$.
Building on the work of Gaspar and Guaraco \cite{GG,G}, Dey obtained the following result.
\begin{thm}[Dey \cite{Dey}]\label{prop:PT-AP}
The $\varepsilon$-phase transition widths and Almgren--Pitts width are related by
$$ h_0^{-1} \lim_{\varepsilon\to 0}\mathbf{L}_{\varepsilon}(\widetilde\Pi) = \mathbf{L}_{\textnormal{AP}}(\Pi).$$ 
\end{thm}

One has an even stronger conclusion. Consider min-max critical points $u_i$ of $E_{\varepsilon_i}$, where $\varepsilon_i \to 0$. In the following, $\mathbf{C}_{\textnormal{PT}}(\widetilde \Pi)$ denotes the set of all limiting stationary integral 1-varifolds arising from $u_i$, while $\mathbf{C}_{\textnormal{AP}}(\Pi)$ denotes the set of all Almgren--Pitts min-max critical points (see Section \ref{subsec:cont-AP}). 
\begin{prop}[Dey {\cite[Theorem 1.4]{Dey}}]\label{thm:CAP.CPT}
$\mathbf{C}_{\textnormal{PT}}(\widetilde \Pi) \subset \mathbf{C}_{\textnormal{AP}}(\Pi)$.
\end{prop}

\begin{rem}
The definitions and the results of this section can be extended to Riemannian manifolds of dimensions $3 \leq n+1 \leq 7$, where we replace geodesics by embedded minimal hypersurfaces.
\end{rem}

\section{The Multiplicity one result}\label{sec:3}

\subsection{Main results}

\begin{thm}\label{thm:mult.one}
Let $(M,g)$ be a closed oriented Riemannian surface of negative sectional curvature. For each $i \in \mathbb{N}$, consider $u_i \in C^\infty(M; [-1,1])$ and $\varepsilon_i > 0$, such that each $u_i$ is a critical point of $E_{\varepsilon_i}$ and
	\begin{equation}
		E_{\varepsilon_i}[u_i] \leq E_0, \quad \ind_{E_{\varepsilon_i}}(u_i) \leq 1 \leq \ind_{E_{\varepsilon_i}}(u_{\varepsilon_i}) + \nul_{E_{\varepsilon_i}}(u_{\varepsilon_i}), \quad \forall\,i \in \mathbb{N}.
	\end{equation}
Suppose $\lim \varepsilon_i = 0$. Passing to a subsequence, write $V = \lim_i h_{0}^{-1} V_{\varepsilon_i}[u_i]$ for the limit $1$-varifold.
Then, there exist a primitive figure eight geodesic $\sigma_{\infty}$ and a possibly empty collection of primitive simple closed geodesics $\sigma_{1},\dots,\sigma_{N}$, which are disjoint from $\sigma_{\infty}$ and pairwise disjoint, such that
\begin{equation*}
V =  \mathbf{v}(\sigma_{\infty},\mathbf{1}_{\sigma_{\infty}}) + \sum_{j=1}^{N} \mathbf{v}(\sigma_{j},\mathbf{1}_{\sigma_{j}}).
\end{equation*}  
Moreover, $\spt V$ separates $M$.
\end{thm}

\begin{proof}
It follows from the work of Mantoulidis \cite{Man} that $\spt(V)$ is a union of primitive closed geodesics such that the number of self-intersection points of the union of these curves is at most one. Moreover the intersection (if occurs) is transverse, so that the respective geodesic is a figure eight curve. Thus we have to prove that each component has multiplicity one and that there is a figure eight on $\spt(V)$. The multiplicity one property implies that $\spt V$ separates $M$.

Assume for a moment that we already know that each simple geodesic on $\spt(V)$ is realized with multiplicity one and let us use that to prove $\spt(V)$ contains a figure eight. Suppose  $\spt(V)$ consists of simple closed geodesics, all attained with multiplicity one. Since
$$\ind_{E_{\varepsilon_i}}(u_{\varepsilon_i}) \leq 1 \leq \ind_{E_{\varepsilon_i}}(u_{\varepsilon_i}) + \nul_{E_{\varepsilon_i}}(u_{\varepsilon_i}),$$
by Theorems \ref{thm:ind.lb} and \ref{thm:ind.ub} we have
$$\ind\big(\spt(V)\big) \leq 1 \leq \ind\big(\spt(V)\big) + \nul\big(\spt(V)\big).$$
However, since the metric has negative curvature, all closed geodesics are strictly stable, which contradicts the previous equation.

Let us now move to the proof of the multiplicity one result. Let $\Sigma$ be one of the components of $\spt(V)$. Then $\Sigma$ is strictly stable and exactly one of the following possibilies occurs: 
\begin{itemize}
\item[(a)] $\Sigma$ is simple;
\item[(b)] $\Sigma$ is a figure eight.
\end{itemize}
In the first case, since $\si$ is embedded and strictly stable, the multiplicity one property follows from the main result in \cite{GMN}. However, \cite{GMN} does not apply if $\si$ is a figure eight. To handle (b) we then adapt the ideas in the proof of Theorem 4.1 in \cite{CM2}. Actually, this approach also applies to (a), so we will deal with this case first, since it will make the argument in case (b) clearer.\\

\noindent
{\bf Case (a):} Given $p \in M$, define define the index concentration scale by
$$\mathcal{R}_i(p) = \inf\big\{r > 0;\, \ind\big(u_{\varepsilon_i};B_{r}(p_{\ast})\big)= 1\big\},$$
and denote
$$\mathring{\si} := \big\{p \in M;\, \liminf_{i \to \infty}\mathcal{R}_i(p) > 0 \big\}.$$
Passing to a subsequence if necessary, it follows from \cite[Lemma 4.17]{Man} that $\mathcal{H}^{0}(\si\setminus\mathring{\si}) \leq 1$.\\

\noindent
{\it Case (a.1): $\si$ is simple and $\mathring{\si} = \si$.}\\

Let $\sigma: \mathbb{S}^1 \to M$ be a primitive unit-speed parametrization of $\si$ and let $n = J\sigma^{\prime}$ be the unit normal determined by $\sigma$.
Since $\Sigma$ is simple and compact there is $\delta > 0$ such that in a tubular neighborhood $U$ of $\Sigma$ we can introduce Fermi coordinates via a map $\Phi: \mathbb{S}^{1}\times(-\delta,\delta) \to M$, $\Phi(q,t) = \exp_{\sigma(\theta)}\big(tn(\theta)\big)$, and the metric $g$ can be written as
\begin{equation}\label{eq:met.Fermi}
g = v(\theta,t)^{2}d\theta^2 + dt^2,
\end{equation}
where $\theta$ is the coordinate determined by a parametrization of $\si$. Moreover, we can suppose the other components of $\spt V$ does not intersect $U$.

Since $\mathring{\si} = \si$, passing to a subsequence if necessary, $u_{\varepsilon_i}$ is stable. Let $\beta \in (0,1)$. As proved in \cite[Lemma 4.11]{Man}, there is $\varepsilon_0 > 0$ such that if $\varepsilon_i \leq \varepsilon_0$, then
\begin{equation*}
\varepsilon |\nabla u_{\varepsilon_i}(p)| \geq \frac{1}{4}\min_{|s| < 1 - \beta/2} W(s) > 0,
\end{equation*}
for all $p \in \{|u_{\varepsilon_i}| < 1 - \beta\}$, and the geodesic curvature of $\{u_{\varepsilon_i} = 0\}$ is uniformly bounded by $\Lambda$. Thus, it follows from the main result in \cite{WW2} (see also \cite{WW1} and \cite[Appendix C]{Man}) that for any $\alpha \in (0,1)$ there is $\varepsilon_{\ast} = \varepsilon_{\ast}(\alpha,\beta,\Lambda) \leq \varepsilon_0$ and a constant $C = C(\alpha,\beta,\Lambda) > 0$ such that for $\varepsilon_i < \varepsilon_{\ast}$ the level sets $\{u_{\varepsilon_i} = 0\}$ are smooth simple closed curves and their geodesic curvature $H(u_{\varepsilon_i})$ satisfy
\begin{equation}\label{eq:mean.curv.bound}
|H(u_{\varepsilon_i})| \leq C\varepsilon_i\big(\log\vert\log\varepsilon_i\vert\big)^{2}.
\end{equation}

As a consequence, $\{u_{\varepsilon_i} = 0\}$ consists of a finite number of components $\Gamma_{i,1},\ldots,\Gamma_{i,Q}$, with $Q$ uniformly bounded independent of $i$, such that $\Gamma_{i,\ell}$ is the graph of a $C^{\infty}$ function $f_{i,\ell}: \si \to \re$, $\forall\, \ell \in \{1,\ldots,Q\}$, meaning 
$$\Gamma_{i,\ell} = \big\{\exp_{\sigma(\theta)}\big(f_{i,\ell}(\sigma(\theta))n(\theta)\big);\, \theta \in \mathbb{S}^{1}\big\}.$$
Let $q \in \Gamma_{i,\ell}$ and for $j\neq \ell$ denote by $d_{i,j}(q)$ the signed distance from $q$ to $\Gamma_{i,j}$. Define
$$D_{i,\ell}(q) = \min_{j\neq\ell}\vert d_{i,j}(q)\vert.$$
In \cite[Proposition 10.1]{WW2} the following strong distance sheets estimates are also proved
\begin{equation}\label{eq:sheets}
\lim_{\varepsilon \to 0}\frac{\exp(-\varepsilon^{-1} D_{i,\ell})}{\varepsilon\big(\log\vert\log\varepsilon\vert\big)^{2}} = 0,\quad \ell = 1,\ldots,Q.
\end{equation}
Hence after possibly reindexing we can write 
$$f_{i,1} < \ldots < f_{i,Q}.$$

By the Constancy Theorem \cite[Theorem 41.1]{Sim},  $\Theta^1(V\mres U,\cdot)$ is constant on $\si$. So, if $\Theta^1(V\mres U,p) = 1$ for some $p \in \si$, then $\si$ is realized with multiplicity one. Suppose $\Theta^1(V\mres U,\cdot) > 1$. Then $Q \geq 2$ for a subsequence of $\{\varepsilon_i\}_{i\in \N}$. It suffices to consider the case $Q = 2$, otherwise in the following argument one takes the botton and top sheets and ignore all the intermediate ones. 

Define $f_i = f_{i,2} - f_{i,1}$ and denote by $H_{\Gamma_{i,\ell}}$ the geodesic curvature of the graph of $f_{i,\ell}$. It is proved in \cite{CM2} that 
\begin{equation} \label{eq:mean_curv_pde}
H_{\Gamma_{i,2}} - H_{\Gamma_{i,1}} = - A\,\mathrm{div}_\si(\widehat{\mathbf{B}} \nabla_\si f_i + f_i \widehat{\mathbf{C}}) + \langle \mathbf{\widehat{D}}, \nabla_\si f_i \rangle_{\si} + \widehat{E}f_i\quad \text{ on } \si,
\end{equation}
with coefficients $A = A_p : \re^2 \to \re$ whose argument is $(f_{i,1},f_{i,2})$ and
\begin{align*} 
\widehat{\mathbf{B}} & = \widehat{\mathbf{B}}_p : \re^2 \times (T_p \si)^2 \to \re, \\
\widehat{\mathbf{C}} & = \widehat{\mathbf{C}}_p, \widehat{\mathbf{D}} = \widehat{\mathbf{D}}_p : \re^2 \times (T_p \si)^2 \to T_p \si, \\
\widehat{E} & = \widehat{E}_p : \re^2 \times (T_p \si)^2 \to \re, 
\end{align*}
whose arguments are $(f_{i,1}, f_{i,2}, \nabla_\si f_{i,1}, \nabla_\si f_{i,2}) \in \re^2 \times (T_p \si)^2$. These coefficients are \emph{uniformly bounded} and satisfy
$$ A \geq \kappa, \quad \langle \widehat{\mathbf{B}}v,v\rangle_\si \geq \kappa \Vert v \Vert_{\si}^{2}, \quad v \in T_p \si, $$
for a fixed $\kappa > 0$, provided
$$ \limsup_{i \to \infty} \big[\Vert f_{i,1} \Vert_{C^1(\si)} + \Vert f_{i,2} \Vert_{C^1(\si)}\big] < \infty. $$

Denote $h_{i,\ell} = f_{i,\ell}\circ\sigma$ and $H_{i,\ell} = H_{\Gamma_{i,\ell}}\circ\sigma$. We conclude $h$ satisfy an equation of the form
\begin{equation} \label{eq:mean_curv_pde2}
H_{i,2} - H_{i,1} = - A(\widehat{B} h_{i}^{\prime} + \widehat{C} h_i )^{\prime} + \widehat{D}h_{i}^{\prime} + \widehat{E}h_i,\ \text{ on } S_{2r},
\end{equation}
with coefficients $A_p : \re^2 \to \re$ whose argument is $(h_{i,1},h_{i,2})$ and
\begin{align*} 
\widehat{B}_p, \widehat{C}_p, \widehat{D}_p, \widehat{E}_p : \re^4 \to \re,
\end{align*}
whose arguments are $(h_{i,1},h_{i,2},h_{i,1}^{\prime},h_{i,2}^{\prime}) \in \re^4$. These coefficients are uniformly bounded and satisfy
$$ A \geq \kappa, \quad \widehat{B} \geq \kappa,$$
for a fixed $\kappa > 0$, provided
\begin{equation}\label{eq:hC1}
\limsup_{i \to \infty}\big[\Vert h_{i,1} \Vert_{C^1} + \Vert h_{i,2} \Vert_{C^1}\big] < \infty.
\end{equation}

Then, it follows from the Harnack inequality that there is a constant $c > 0$, independent of $i$, such that, 
$$\sup_{\mathbb{S}^1} h_i \leq c \inf_{\mathbb{S}^1} h_i,\ \forall\, i \in \N.$$
Moreover, \eqref{eq:mean.curv.bound} and \eqref{eq:sheets} imply 
$$\lim_{i \to \infty}\frac{\vert\vert H_{i,\ell}\vert\vert_{C^{0}(\mathbb{S}^1)}}{\varepsilon\vert\log\varepsilon\vert} = 0, \quad \liminf_{i \to \infty}\frac{\inf_{\mathbb{S}^1}h_i}{\varepsilon\vert\log\varepsilon\vert} > 0.$$

Then we can divide \eqref{eq:mean_curv_pde} by $\sup_{\si} h_i$ to obtain 
\begin{align}\label{eq:equationh}
&\int_{\mathbb{S}^1}\frac{H_{i,2} - H_{i,1}}{A\cdot\sup \widehat{h}_i}\,\psi\,d\theta = \int_{\mathbb{S}^1}\big(\widehat{B}\,\widehat{h}_{i}^{\,\prime}\psi^{\prime} + \widehat{C}\, \widehat{h}_{i}\psi^{\prime} + A^{-1}\widehat{D}\,\widehat{h}_{i}^{\,\prime}\psi + A^{-1}\widehat{E}\,\widehat{h}_{i}\psi\big)d\theta,	
\end{align}
where $\widehat{h}_i = \frac{h_i}{\sup h_i}$ and $\psi \in C^{\infty}(\mathbb{S}^1)$. Choosing $\psi = \widehat{h}_i$ we conclude
$$\int_{\mathbb{S}^1}\vert \widehat{h}_{i}^{\,\prime}\vert^{2}d\theta < \infty.$$ 
Since $\widehat{h}_i$ is bounded it follows that
$$\widehat{h}_i \rightharpoonup \widehat{h} \ \text{in}\  W^{1,2}(\mathbb{S}^1)\quad \text{and}\quad \widehat{h}_i \rightharpoonup \widehat{h}\ \text{in}\  L^{2}(\mathbb{S}^1).$$

Passing to the limit in \eqref{eq:equationh}, we conclude $\widehat{h}$ is a weak solution of 
$$\varphi^{\prime\prime} + (K_{g}\circ\sigma)\varphi = 0,$$
and then it is smooth, by elliptic regularity. Moreover, 
$$\frac{1}{c} \leq \widehat{h} \leq 1.$$
Thus, $\widehat{h}$ is a non-trivial jacobi field of $\si$. However, this is a contradiction, since $\si$ is strictly stable. Therefore $\Theta^{1}(V\mres U,\cdot) = 1$.\\

\noindent
{\it Case (a.2): $\si$ is simple and $\si\setminus\mathring{\si} = \{p\}$.}\\

The convergence of $\{u_{\varepsilon_i}=0\}\cap U$ to $\si\setminus\{p\}$ is graphical $C^{2,\alpha}_{loc}$ on $\si\setminus\{p\}$. Consider $0 < \rho < \inj(M,g)$. For all $r \in (0,\rho/4)$ there exists a subsequence along which
\begin{equation*}
\ind\big(u_{\varepsilon_i};M\backslash B_{r}(p_{\ast})\big) = 0,\ \forall\, i.
\end{equation*}
Denote $D_{R} := \mathbb{S}^{1}\setminus\sigma^{-1}\big(\si\setminus\overline{B}_{R}(p)\big)$. The reasoning of the previous case still applies and yields functions $h_{i,1},h_{i,2}: D_{2r} \to \re$ associated to the incomplete properly embedded surfaces (with boundary) 
$$\Gamma_{i,1},\Gamma_{i,2} \subset \big(\{u_{\varepsilon_i}=0\}\cap U\big) \setminus \big\{\exp_{\sigma(\theta)}\big(tn(\theta)\big);\, \theta \in \mathbb{S}^1, t \in (-\delta,\delta)\big\}.$$

As before, we conclude that there is a constant $c > 0$ depending on $r$ but independent of $i$, such that, 
$$\sup_{D_{2r}} h_i \leq c \inf_{D_{2r}} h_i,\ \forall\, i \in \N,$$
and also
\begin{equation}\label{eq:bounds.h}
\lim_{i \to \infty}\frac{\vert\vert H_{i,\ell}\vert\vert_{C^{0}(D_{2r})}}{\varepsilon\vert\log\varepsilon\vert} = 0, \quad \liminf_{i \to \infty}\frac{\inf_{D_{2r}}h_i}{\varepsilon\vert\log\varepsilon\vert} > 0.
\end{equation}

Let $\psi: \mathbb{S}^1 \to \re$ be a smooth function whose support is compact and it is contained in $D_{3r}$. Then, by \eqref{eq:mean_curv_pde2} we have
\begin{align*} 
&\int_{D_{3r}}\frac{H_{i,2} - H_{i,1}}{A\cdot\sup \widehat{h}_i}\,\psi\,d\theta = \int_{D_{3r}}\big(\widehat{B}\,\widehat{h}_{i}^{\,\prime}\psi^{\prime} + \widehat{C}\, \widehat{h}_{i}\psi^{\prime} + A^{-1}\widehat{D}\,\widehat{h}_{i}^{\,\prime}\psi + A^{-1}\widehat{E}\,\widehat{h}_{i}\psi\big)d\theta.	
\end{align*}
Consider $\eta: \mathbb{S}^1 \to \re$, such that $0 \leq \eta \leq 1$, $\eta\vert_{D_{4r}} \equiv 1$ and $\vert\eta^{\prime}\vert \leq \frac{C}{r}$, for some constant $C > 0$. Since the functions $\widehat{B},\,\widehat{C},\,\widehat{D},\,\widehat{E}$ are uniformly bounded with $\widehat{B} > \kappa >0$ and it holds \eqref{eq:hC1}, applying $\psi = \eta \widehat{h}_i$ to the equality above we obtain
$$\limsup_{i \to \infty}\int_{D_{4r}}(\widehat{h}_{i}^{\,\prime})^{2}d\theta < \infty.$$

Moreover, since $\Gamma_{i,\ell}$ converge in $C^{2,\alpha}$ to $\si\backslash B_{r}(p)$, the coefficients of \eqref{eq:mean_curv_pde2} satisfy
\begin{align*}
\limsup_{r \to 0}\Big[\limsup_{i \to \infty} &\big(\vert\vert A\vert\vert_{C^0(D_{3r})} + \vert\vert \widehat{B}\vert\vert_{C^0(D_{3r})} + \vert\vert \widehat{C}\vert\vert_{C^0(D_{3r})}\\
&\vert\vert \widehat{D}\vert\vert_{C^0(D_{3r})} + \vert\vert \widehat{E}\vert\vert_{C^0(D_{3r})}\big)\Big] < \infty.
\end{align*}
Hence
\begin{equation}\label{eq:h.der.L2}
\limsup_{r \to 0}\left[\limsup_{i \to \infty}\int_{D_{4r}}(\widehat{h}_{i}^{\,\prime})^{2}d\theta\right] < \infty.
\end{equation}
Therefore,
$$\widehat{h}_{i} \rightharpoonup \widehat{h}\ \text{in}\ W^{1,2}_{\mathrm{loc}}(\mathbb{S}^{1}\backslash\sigma^{-1}(p)), \quad \widehat{h}_{i} \to \widehat{h}\ \text{in}\ L^{2}_{\mathrm{loc}}(\mathbb{S}^{1}\backslash\sigma^{-1}(p)).$$

We conclude $\widehat{h}$ is a weak solution on $\mathbb{S}^{1}\backslash\sigma^{-1}(p)$ of the equation
\begin{equation}\label{eq:Jacobi.field}
\widehat{h}^{\,\prime\prime} + (K_g\circ\sigma) \widehat{h} = 0.
\end{equation}
By elliptic regularity, $\widehat{h}$ is smooth and
solves \eqref{eq:Jacobi.field} classically on $\mathbb{S}^{1}\backslash\sigma^{-1}(p)$.
We are going to prove on Lemma \ref{lem.Linfty} the following property: 
\begin{equation}\label{eq:h.bounded}
\widehat{h} \in L^{\infty}(\mathbb{S}^1).
\end{equation}
	
It follows from \eqref{eq:h.der.L2}, \eqref{eq:Jacobi.field} and \eqref{eq:h.bounded} that $\widehat{h}$ satisfies
$$\int_{\mathbb{S}^1}\big[(\widehat{h}^{\,\prime})^2 - (K_g\circ\sigma) \widehat{h}^2\big]d\theta = 0.$$
However, this contradicts the fact $K_g < 0$. Therefore $\Theta^{1}(V\mres U,\cdot) = 1$.\\

\noindent
{\bf Case (b)}: Let $p_{\ast}$ be the point of self-intersection of $\Sigma$. Observe
$$\mathrm{\sing} \spt \vert\vert V\vert\vert = \{p_{\ast}\}.$$ 
Define
$$\mathcal{R}_i = \inf\big\{r > 0;\, \ind\big(u_{\varepsilon_i};B_{r}(p_{\ast})\big)= 1\big\}.$$
We claim 
\begin{equation}\label{eq:lim_Ri}
\liminf_{i \to \infty} \mathcal{R}_i = 0.
\end{equation}
Suppose $\liminf_{i \to \infty} \mathcal{R}_i = 2r_{\ast} > 0$. Then, passing to a subsequence if necessary, we have $\ind\big(u_{\varepsilon_i};B_{r_{\ast}}(p_{\ast})\big)= 0$, $\forall\, i$. However this contradicts \cite[Theorem 4.12]{Man}, since $\mathrm{\sing} \spt \vert\vert V\vert\vert\mres B_{r}(p_{\ast}) \neq \emptyset$.

Consider $0 < \rho < \mathrm{inj}(M,g)$ such that $\si\cap B_{\rho}(p_{\ast})$ is homeomorphic to the union of two embedded arcs intersecting only at $p_{\ast}$.
For a fixed $0 < r < \rho/8$, by \eqref{eq:lim_Ri} and \cite[Lemma 4.16]{Man} there is a subsequence along which it holds
\begin{equation}\label{eq:stable.cond1}
\ind\big(u_{\varepsilon_i};M\backslash B_{r}(p_{\ast})\big) = 0,\ \forall\, i.
\end{equation}
Then, by \cite[Lemma 4.11]{Man}
\begin{equation}\label{eq:stable.cond2}
\varepsilon_{i}\vert\nabla u_{\varepsilon_{i}}\vert \geq \frac{1}{4}\min_{|s| < 1 - \beta/2} W(s)\ \text{in}\ \big\{\vert u_{\varepsilon_i}\vert \leq 1 -\beta\big\}\backslash \overline B_{2r}(p_\ast),\ \forall\, i.
\end{equation}

The idea now is similar to the proof of the previous case. However one complication arises due to the fact $\si$ does not have an embedded tubular neighborhood where Fermi coordinates can be defined. To handle this, we will break $\si\backslash B_{r}(p_{\ast})$ in six pieces and introduce Fermi coordinates with respect to each such subcurve. Now in each piece we describe a portion of $\{u_i = 0\}$ as a graph. Composing the functions with a parametrization of $\si$ we will then see $\{u_i = 0\}$ as a parametrized graph.

As before, let $\sigma: \mathbb{S}^1 \to M$ be a primitive unit-speed parametrization of $\si$ and let $n = J\sigma^{\prime}$ be the unit normal determined by $\sigma$. Consider $\delta > 0$, such that the map $\Phi: \mathbb{S}^1\times (-\delta,\delta) \to U \subset M$, $\Phi(\theta,t) = \exp_{\sigma(\theta)}\big(tn(\theta)\big)$, is an immersion. We can choose $U$ which does not intersect the other components of $\spt V$.

%There exist $\theta_1,\theta_2 \in \mathbb{S}^1$ such that $\sigma(\theta_1) = \sigma(\theta_2) = p_\ast$ and $\sigma: \mathbb{S}^{1}\backslash\{\theta_1,\theta_2\} \to \si\backslash\{p_\ast\}$ is injective. So, by abuse of notation, in the following we use $\theta$ to denote the points of $\si\backslash\{p_\ast\}$.

\begin{figure}[!htb]
\centering
\includegraphics[scale=0.09]{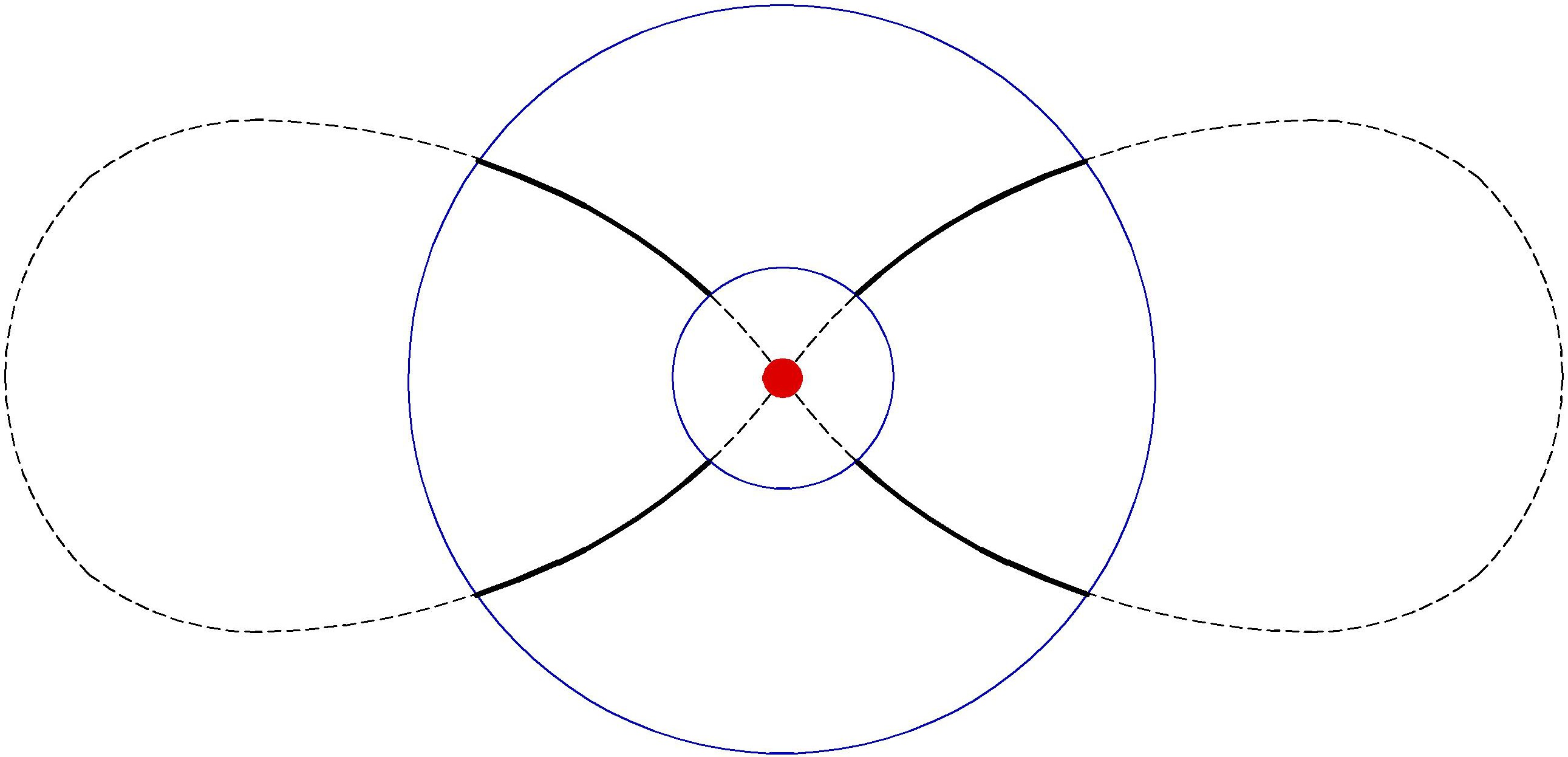}
\caption{$\si\cap\big(B_{\rho}(p_{\ast})\backslash \overline B_{2r}(p_{\ast})\big)$.}
\end{figure}

Observe that $\si\cap\big(B_{\rho}(p_{\ast})\backslash \overline B_{2r}(p_{\ast})\big)$ is an embedded curve in $M$, which has four connected components $\si^{\lambda}$, $\lambda=1,2,3,4$. Taking a smaller $\delta$ if necessary, the map $\Phi$ above define Fermi coordinates $\si^{\lambda}\times(-\delta,\delta) \to M$, $\forall\,\lambda$, such that the metric $g$ assumes the form \eqref{eq:met.Fermi}. Observe that the sets $\Phi\big(\si^{\lambda}\times(-\delta,\delta)\big)$ and $\Phi\big(\si^{\nu}\times(-\delta,\delta)\big)$ possibly have non empty intersection for $\lambda \neq \nu$. However, for each fixed $\lambda$, $\Phi\big(\si^{\lambda}\times(-\delta,\delta)\big)$ is an embedded two-dimensional submanifold of $M$.

By \eqref{eq:stable.cond1}, \eqref{eq:stable.cond2} and \cite[Lemma 4.11]{Man} passing to a subsequence if necessary we can apply the results in \cite{WW2} to conclude that $\{u_{\varepsilon_i} = 0\}\cap \big(B_{\rho}(p_{\ast})\backslash \overline B_{2r}(p_{\ast})\big)$ consists of a finite number of smooth simple curves $\Gamma_{i,1},\ldots,\Gamma_{i,Q}$, with $Q$ uniformly bounded independent of $i$, such that the bound \eqref{eq:mean.curv.bound} holds and $\Gamma_{i,\ell}$ has four connected components $\Gamma_{i,\ell}^{\lambda}$, $\lambda=1,2,3,4$. Moreover, $\Gamma_{i,\ell}^{\lambda}$ is the graph of a $C^{\infty}$ function $f_{i,\ell}^{\lambda}: \si^{\lambda} \to \re$  (with respect to the Fermi coordinates as above), such that \eqref{eq:sheets} holds. So we can assume,
$$f_{i,1}^{\lambda} < \ldots < f_{i,Q}^{\lambda}.$$
For each $\ell=1,\ldots,Q$ and each $\lambda=1,2,3,4$, we consider the composition $f_{i,\ell}^{\lambda}\circ\sigma: \sigma^{-1}(\si^{\lambda}) \to \re$.

\begin{figure}[!htb]
\centering
\includegraphics[scale=0.09]{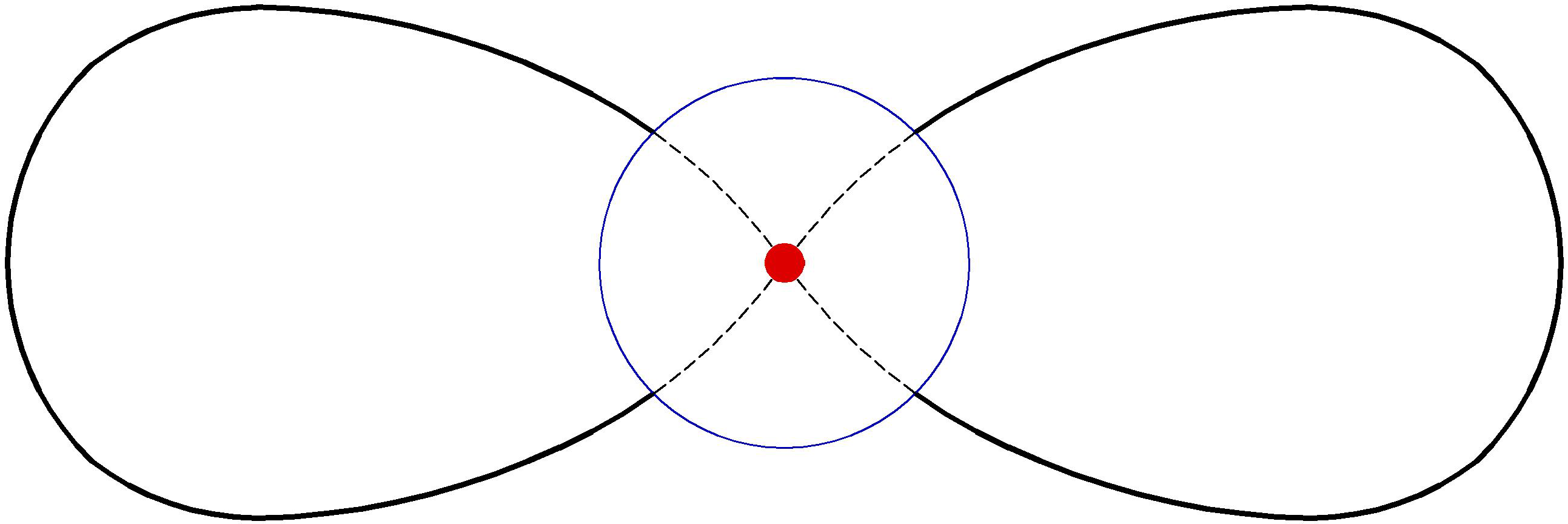}
\caption{$\si\backslash \overline B_{\rho/2}(p_{\ast})$.}
\end{figure}

On the other hand, $\si\backslash \overline B_{\rho/2}(p_{\ast})$ consists of two connected components $\widetilde{\si}^{\mu}$, $\mu = 1,2$, and we can also introduce Fermi coordinates $\widetilde{\si}^{\mu}\times(-\delta,\delta) \to M$.
Also, the properties \eqref{eq:stable.cond1} and \eqref{eq:stable.cond2} hold with $r$ replaced by $\rho/6$. So, we can apply again the results in \cite{WW2} to conclude $\{u_{\varepsilon_i} = 0\}\cap \big(M\backslash \overline B_{\rho/2}(p_{\ast})\big)$ consists of a finite number of smooth simple curves $\widetilde{\Gamma}_{i,1},\ldots,\widetilde{\Gamma}_{i,\widetilde{Q}}$, with $\widetilde{Q}$ uniformly bounded independent of $i$, such that the bound \eqref{eq:mean.curv.bound} holds, and $\widetilde{\Gamma}_{i,\ell}$ has two connected componets $\widetilde{\Gamma}_{i,\ell}^{\mu}$, $\mu=1,2$. Moreover, $\widetilde{\Gamma}_{i,\ell}^{\mu}$
is the graph of a $C^{\infty}$ function $\widetilde{f}_{i,\ell}:\widetilde{\si}^{\mu} \to \re$ and \eqref{eq:sheets} holds, hence we can suppose,
$$\widetilde{f}_{i,1}^{\mu} < \ldots < \widetilde{f}_{i,\widetilde{Q}}^{\mu}.$$
As before we consider the composition $\widetilde{f}_{i,\ell}\circ\sigma: \sigma^{-1}(\widetilde{\si}^{\mu}) \to \re$.

We can assume $Q = \widetilde{Q}$ and reindexing if necessary, it follows that the curves
$\widetilde{\Gamma}_{i,\ell}$ and $\Gamma_{i,\ell}$ coincide in the set $\si\cap\big(B_{\rho}(p_{\ast})\backslash \overline B_{\rho/2}(p_{\ast})\big)$. Thus, we can join the functions $f_{i,\ell}^{\lambda}\circ \sigma$ and $\widetilde{f}_{i,\ell}^{\mu}\circ \sigma$ in order to obtain a smooth function $h_{i,\ell}: S_{2r} \to \re$, where from now on we denote $S_{R} := \mathbb{S}^{1}\backslash\sigma^{-1}\big(\si\backslash \overline B_{R}(p_{\ast})\big)$.

Suppose $\Theta^{1}(V\mres U,\cdot) > 1$. We will proceed as in the proof of case (a), assuming $Q = 2$. Denote $h_i = h_{i,2} - h_{i,1}$. Let $H_{i,\ell}: S_{2r} \to \re$ be the geodesic curvature of the parametrized curve $\theta \mapsto \exp_{\sigma(\theta)}\big(h_{i,\ell}(\theta)n(\theta)\big)$.

Observe the functions $f_i := f_{i,2}^{\lambda} - f_{i,1}^{\lambda}$ and $\widetilde{f}_i := \widetilde{f}_{i,2}^{\mu} - \widetilde{f}_{i,1}^{\mu}$ satisfy an equation of the same format as \eqref{eq:mean_curv_pde}. Writing each such equation using the parametrization $\sigma$, we conclude $h$ satisfy \eqref{eq:mean_curv_pde2} with the arguments satisfying the same properties as in case (a.1).

Now, the proof follows exactly as in case (a.2) replacing $D_R$ by $S_R$. So, at the end we obtain $\widehat{h} \in C^{\infty}\big(\mathbb{S}^{1}\backslash\sigma^{-1}(p_{\ast})\big)\cap L^{\infty}(\mathbb{S}^1)$ (by elliptic regularity and Lemma \ref{lem.Linfty}) satisfying
$$\int_{\mathbb{S}^1}\big[(\widehat{h}^{\,\prime})^2 - (K_g\circ\sigma) \widehat{h}^2\big]d\theta = 0.$$
This is a contradiction and so $\Theta^{1}(V\mres U,\cdot) = 1$.
\end{proof}

\begin{rem}
Since the domain of $\widehat{h}$ has dimension $1$, a priori there is no removable singularity result available. So in cases (a.2) and (b) not necessarily $\widehat{h}$ is a classical solution of the Jacobi equation in all of $\mathbb{S}^{1}$.
\end{rem}

\begin{dem1}
By Lemma \ref{lem:bound.dim} there exist a sequence $\{X_i\}_{i\in \N}$ of cubical subcomplexes with dimension bounded independent of $i$, and a sequence of homotopy classes $\{\Pi_i\}_{i\in \N}$ of $\mathbf{F}$-continuous sweepouts $\Phi: X_i \to \mathcal{Z}_1(M;\mathbb{Z}_2)$, such that 
\begin{align*}
 \lim_{i \to \infty} \mathbf{L}_{\textrm{AP}}(\Pi_i) &= \omega_1(M,g).
\end{align*}

Moreover, by Theorem \ref{prop:PT-AP},
$$h_0^{-1} \lim_{\varepsilon\to 0}\mathbf{L}_{\varepsilon}(\widetilde\Pi_i) = \mathbf{L}_{\textnormal{AP}}(\Pi_i) =: \mathbf{L}_i.$$
It follows from the previous theorem that there exist $V_i \in \mathbf{C}_{\mathrm{PT}}(\widetilde\Pi_i)$ of the form
\begin{equation}\label{eq:CPT}
V_i = \mathbf{v}(\sigma_{i,\infty},\mathbf{1}_{\sigma_{i,\infty}}) + \sum_{j=1}^{N_i} \mathbf{v}(\sigma_{i,j},\mathbf{1}_{\sigma_{i,j}}),
\end{equation}  
where $\sigma_{i,\infty}$ is a primitive figure eight geodesic, $\sigma_{i,1},\dots,\sigma_{i,N_i}$ are primitive simple closed geodesics, and all the curves are pairwise disjoint.

Since $\mathbf{L}_{\textrm{AP}}(\Pi_i) \to \omega_1(M,g)$, the mass of the varifolds in the sequence $\{V_i\}_i$ is uniformly bounded, thus a subsequence of $\{V_i\}_i$ converge as varifolds to an integral varifold $V$ and the convergence is smooth outside $\sing V$. 
On the other hand, since $\length(\sigma_{i,j}) \geq 2\,\mathrm{inj}(S)$, $\forall\, (i,j)$, the sequence $\{N_{i}\}_i$ is uniformly bounded, and hence it is finite. Also, it is well known that in a hyperbolic surface for any $L>0$ there exists only a finite number of closed geodesics with length at most $L$. 
Therefore passing again to a subsequence if necessary we have 
$$V_i = V = \mathbf{v}(\sigma_{\infty},\mathbf{1}_{\sigma_{\infty}}) + \sum_{j=1}^{N} \mathbf{v}(\sigma_{j},\mathbf{1}_{\sigma_{j}}),$$ where $\sigma_{\infty}$ is a figure eight geodesic, and $\sigma_{j},\, j=1,\ldots,N$, are simple closed geodesics, and all the curves are pairwise disjoint. 

The fact $N \leq 3m -3$ follows from the Collar Theorem \cite[Theorem 4.1.1]{Bus}, since $\sigma_{1},\ldots,\sigma_{N}$ are disjoint simple closed geodesics.

Finally, by a result of Buser \cite[Theorems 4.2.2 and 4.3.1]{Bus}, the figure eight $\sigma_{\infty}$ satisfies
$$\length(\sigma_{\infty}) > 2\big(-\inf_{M}K_g\big)^{-1/2}\mathrm{arccosh}\,3.$$
This finishes the proof of the theorem.
\end{dem1}

\subsection{An auxiliar Lemma}

\begin{lem}\label{lem.Linfty} Let $\widehat{h}$ be as in the cases (a.2) and (b) in the proof of Theorem \ref{thm:mult.one}. Then
$\widehat{h} \in L^{\infty}(\mathbb{S}^1)$ and $\widehat{h} \neq 0$ a.e on $\mathbb{S}^1$.
\end{lem}

We are going to describe the argument for the case (b), the other one is similar.
The proof follows along the same lines as the proof of \cite[Proposition 4.3]{CM2}. We will give some of the details below. One of the main ingredients of the proof in \cite{CM2} is a result due to White \cite[Appendix]{Wh}. In the original paper this result is only stated for funtionals defined in a ball of $\mathbb{R}^2$, however a inspection of its proof reveals that it works in every dimension, so the following version also holds.

\begin{prop}\label{prop:foliation}
Let $\Phi$ be an even elliptic integrand, where $\Phi$ and $D_2 \Phi$ are $C^{2,\alpha}$. Let $\Phi_r$ be the integrand defined by $\Phi_r(x, v) = \Phi(rx, v)$. There is an $\eta > 0$ such that if $r < \eta$ and if
$$w : [-1,1] \to \mathbb{R}, \quad \Vert w \Vert_{C^{2,\alpha}} < \eta,$$
then for each $t \in [-1,1]$, there is a $C^{2,\alpha}$ function $v_t : [-1,1] \to \mathbb{R}$ whose graph is $\Phi_r$-stationary and such that 
$$v_t(\theta) = w(\theta) + t,\quad \text{ if } \theta \in \{-1,1\}.$$
Moreover, $v_t$ depends in a $C^1$ way on $t$ so that the graphs of the $v_t$ foliate a region of $\mathbb{R}^2$. If $\Gamma$ is a $C^1$ properly immersed $\Phi_{r}$-stationary curve in $B_{1/2}(0)$ with $\partial \Gamma \subset \mathrm{graph}\, v_t$, then $\Gamma \subset \mathrm{graph}\, v_t$. 
\end{prop} 

Consider a function $w_{i}^{\lambda}: (-2\rho,2\rho) \to \re$ such that $(w_{i}^{\lambda})^{\prime\prime} = 0$ and $w_{i}^{\lambda}(\pm \rho) = f_{i,2}^{\lambda}(\pm \rho)$. Since $f_{i,2}^{\lambda} \to 0$ in $C^{2,\alpha}\big(B_{\rho}(p_{\ast})\backslash \overline B_{\rho/2}(p_{\ast})\big)$, taking a subsequence if necessary, we can suppose that a suitable rescaling of $w_{i}^{\lambda}$ is under the hyphoteses of Proposition \ref{prop:foliation}. Then we obtain a foliation $t \mapsto \si_{i}^{\lambda}(t)$, $t \in [-\delta,\delta]$, composed of geodesics which are graphs over $\Sigma\cap B_{\rho}(p_{\ast})$. We can choose $\delta > 0$ such that $\cup_{t\in(-\delta,\delta)}\si_{i}^{\lambda}(t) \subset U$. Moreover, the geodesic curvature of $\si_{i}^{\lambda}(t)$ satisfy
$$\big\Vert H_{\si_{i}^{\lambda}(t)}\big\Vert_{C^{0,\alpha}} \leq \eta,$$
where $\eta > 0$ can be made arbitrarily small.

Now, the idea is to use the foliation to construct solutions of the Allen-Chan equation which will be used as barriers to bound $f_{i,2}$ from above. 

Denote $\Omega = (-\rho,\rho)\times \left[-\frac{1}{2},\frac{1}{2}\right]$. Consider Fermi coordinates around $\si_{i}^{\lambda}(t)$ defined by its normal exponential map, denoted $\Phi_{i,t}^{\lambda}: (-\rho,\rho)\times \left[-\frac{1}{2},\frac{1}{2}\right] \to M$. Define $g_{i}^{\lambda}(t)$ as the pullback of $g$ by $\Phi_{i,t}^{\lambda}$.

Let $\mathbb{H}$ be the heteroclinic solution of \eqref{eq:ACeqt} with $\varepsilon=1$, defined on $\re$. Define $\mathbb{H}_{\varepsilon}(s) = \mathbb{H}(\varepsilon^{-1}s)$. We will use some functions defined in \cite{CM1}:
\begin{itemize}
\item[(a)] for a fixed $\delta_{\ast} \in (0,1)$ and $j \in \mathbb{N}$, consider $\chi_j:\re \to [0,1]$ satisfying $\chi_{j}^{\prime} \geq 0$ on $[0,\infty)$ and such that
$$\sum_{k=0}^{3}\varepsilon^{k\delta_{\ast}}\Vert \nabla^{k}\chi_{j}\Vert_{L^{\infty}} \leq 200, \quad \chi_{j}(s) = \left\{
\begin{array}{rl}
1, & |s| \leq \varepsilon^{\delta_{\ast}}\left(1 - \frac{2j - 1}{100}\right),\\\\
0, &  |s| \geq \varepsilon^{\delta_{\ast}}\left(1 - \frac{2j - 2}{100}\right);
\end{array} \right.$$
\item[(b)] $\widetilde{\mathbb{H}}_{\varepsilon}$  is a truncated version of $\mathbb{H}_{\varepsilon}$ defined by 
$$\widetilde{\mathbb{H}}_{\varepsilon}(s) = \chi_{1}(s)\mathbb{H}_{\varepsilon} \pm \big(1 - \chi(s)\big),$$
where the $\pm$ corresponds to $t > 0$ and $t < 0$, respectively.
\item[(c)] $\hat{\chi}$ is a cutoff function such that
$$\hat{\chi}(s) = \left\{
\begin{array}{rl}
1, & |s| \leq B\varepsilon\vert\log\varepsilon\vert,\\\\
0, &  \vert s\vert \geq 2B\varepsilon\vert\log\varepsilon\vert,
\end{array} \right.$$
for a constant $B >> 1$ which will be chosen later, and also 
$$\vert\hat{\chi}^{(k)}\vert = O\big((\varepsilon\vert\log\varepsilon\vert)^{-k}\big),\quad \text{for}\ k \geq 1,\, \varepsilon \to 0;$$
\item[(d)] define
$$\hat{v}^{\sharp}(s) = \kappa \hat{\chi}(s)\mathbb{H}^{\prime}(\varepsilon^{-1}s) + \big(1 - \hat{\chi}(s)\big)\left\{
\begin{array}{rl}
1 - \varepsilon^{3} - \widetilde{\mathbb{H}}_{\varepsilon}(s), & s > 0,\\\\
-1 - \widetilde{\mathbb{H}}_{\varepsilon}(s), &  s < 0,
\end{array} \right.$$ 
where $\kappa \in \re$ is chosen so that the following holds
$$\int_{-\infty}^{\infty}\hat{v}^{\sharp}(s)\mathbb{H}^{\prime}(\varepsilon^{-1}s)\,ds = 0;$$
we choose the constant $B$ of item (c) to be sufficiently large so that (\cite[page 259]{CM1})
$$\sum_{k=0}^{2}\varepsilon^{k\delta_{\ast}}\Vert \nabla^{k}\hat{v}^{\sharp}\Vert_{L^{\infty}} + \varepsilon^{2+\alpha}[\hat{v}^{\sharp}]_{\alpha} = O(\varepsilon^3).$$
\item[(e)] for $(\theta,s) \in \pa\Omega$, define
$$\hat{v}^{\flat}(\theta,s) = \big(1 - \chi_{4}(s)\big)\left\{
\begin{array}{rl}
1 - \varepsilon^{3} - \widetilde{\mathbb{H}}_{\varepsilon}(s), & s > 0,\\\\
-1 - \widetilde{\mathbb{H}}_{\varepsilon}(s), &  s < 0.
\end{array} \right.$$
\end{itemize}

One can then apply Theorem 7.4 in \cite{CM1} to find a function $\mathfrak{u}_{i,t}^{\lambda}: \Omega \to \re$ such that
$$\varepsilon_i\Delta_{g_{i}^{\lambda}(t)}\mathfrak{u}_{i,t}^{\lambda} = W^{\prime}(\mathfrak{u}_{i,t}^{\lambda})$$
and, for all $(\theta,s) \in \pa\Omega$,
$$\mathfrak{u}_{i,t}^{\lambda} = \widetilde{\mathbb{H}}_{\varepsilon}(s) + \chi_{4}(s)\hat{v}^{\sharp}(s) + \hat{v}^{\flat}(\theta,s).$$  

Now, proceeding exactly as in \cite[subsection 4.2]{CM1}, one can prove that for an appropriate choice of the constant $B$, there exist exactly one $\tau_{i}^{\lambda} \in (-\delta,\delta)$ and at least one $q_{i}^{\lambda}$, such that
\begin{enumerate}
\item $\mathfrak{u}_{i,t}^{\lambda} < u_{\varepsilon_i}\circ \Phi_{i,t}^{\lambda}$ on $\Omega$, for all $t \in (\tau_{i}^{\lambda},\delta]$,
\item $\mathfrak{u}_{i,\tau_{i}^{\lambda}}^{\lambda}(q_{i}^{\lambda}) =   u_{\varepsilon_i}\circ \Phi_{i,\tau_{i}^{\lambda}}^{\lambda}(q_{i}^{\lambda})$,
\item $0 \leq \tau_{i}^{\lambda} < 7B\varepsilon_i\vert\log\varepsilon_i\vert$.
\end{enumerate}
This implies
$$f_{i,2}^{\lambda}(\theta) \leq h_{i,\tau_{i}^{\lambda}}^{\lambda}(\theta) \leq h_{i,7B\varepsilon_i\vert\log\varepsilon_i\vert}^{\lambda}(\theta),\quad \text{for}\ 2r\leq \vert\theta\vert < 2\rho,$$
where $h_{i,t}^{\lambda}$ denotes the height of the geodesic $\si_{i}^{\lambda}(t)$ over $\si$.

One can also apply Proposition \ref{prop:foliation} as above using the boundary data of $f_{i,1}^{\lambda}$ to obtain a foliation by geodesics $\widetilde{\si}_{i}^{\lambda}(t)$, whose height over $\si$ is denoted by $\widetilde{h}_{i,t}^{\lambda}$. By a similar argument one can then conclude
$$f_{i,1}^{\lambda}(\theta) \geq \widetilde{h}_{i,7B\varepsilon_i\vert\log\varepsilon_i\vert}^{\lambda}(\theta),\quad \text{for}\ 2r\leq \vert\theta\vert < 2\rho.$$
Hence, by the properties of the foliations and the maximum principle, on $(-2\rho,-2r)\cup(2r,2\rho)$ we have
\begin{align*}
f_{i,2}^{\lambda} - f_{i,1}^{\lambda} &= h_{i,7B\varepsilon_i\vert\log\varepsilon_i\vert}^{\lambda} - \widetilde{h}_{i,7B\varepsilon_i\vert\log\varepsilon_i\vert}^{\lambda}\\ 
&\leq c\left(7B\varepsilon_i\vert\log\varepsilon_i\vert + h_{i,0}^{\lambda} - \widetilde{h}_{i,0}^{\lambda}\right)\\
&\leq \widetilde{c}\left(\varepsilon_i\vert\log\varepsilon_i\vert + \max_{[-\rho,\,\rho]}(h_{i,0}^{\lambda} - \widetilde{h}_{i,0}^{\lambda})\right)\\ 
&\leq \widetilde{c}\left(\varepsilon_i\vert\log\varepsilon_i\vert + \max_{\{-\rho,\,\rho\}}(f_{i,2}^{\lambda} - f_{i,1}^{\lambda})\right),
\end{align*}
where $\widetilde{c}$ is independent of $\lambda$, $i$ and $r$.

It follows that
$$\sup_{S_{2r}}h_{i} \leq \widetilde{c}\left(\varepsilon_i\vert\log\varepsilon_i\vert + \sup_{S_{\rho}}h_{i}\right).$$
Now, on right side of the last inequality we use \eqref{eq:bounds.h} on the first term and the Harnack inequality on the second term to conclude
$$\sup_{S_{2r}}h_{i} \leq c^{\prime}\inf_{S_{\rho}}h_{i},$$
and this holds independently of $i$ and $r$. Taking the limit first as $i \to \infty$ and then as $r \to 0$ we conclude $\widehat{h}$ satisfies the conclusion of the lemma.

\section{Applications for hyperbolic surfaces}\label{sec:4}

Along this section $S$ denotes a closed oriented hyperbolic surface. The geodesics of $S$ will always be parametrized by arc-length. In the following we will often use the same notation for a curve and its parametrization.

\subsection{A sweepout of a pair of pants}\label{sec:pants}

\begin{figure}[!htb]
\centering
\includegraphics[scale=0.08]{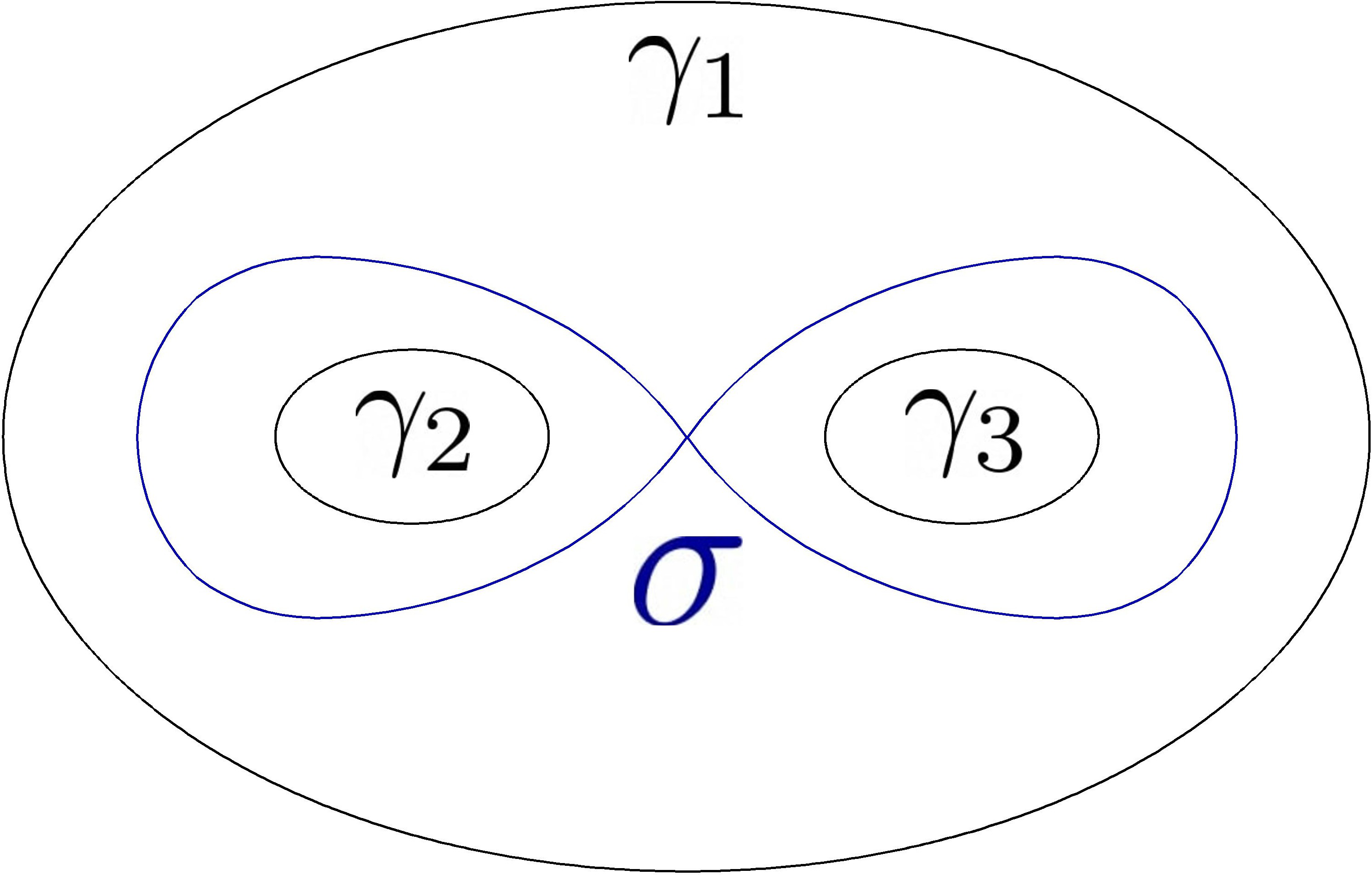}
\caption{A figure eight geodesic in a pair of pants.}
\label{fig3}
\end{figure}

Let $\sigma \subset S$ be a figure eight geodesic. Then, $\sigma$ is contained in a pair of pants $\mathrm{Y}$ \cite[Theorems 4.2.2 and 4.2.4]{Bus}. Let $\gamma_1$, $\gamma_2$ and $\gamma_3$ be the boundary geodesics of $\mathrm{Y}$, labelled accordingly to figure \ref{fig3}. Denote $L_{\sigma} = \length(\sigma)$ and $L_i = \length(\gamma_i)$, $i=1,2,3$. By \cite[Theorem 4.2.2]{Bus} we have
\begin{equation}\label{eq:length.fig.8}
\begin{split}
L_{\sigma} &= 2\,\mathrm{arcosh}\big[\cosh \left(L_3/2\right) + 2\cosh \left(L_1/2\right)\cosh \left(L_2/2\right)\big]\\
&> 2\,\mathrm{arcosh}\,3.
\end{split}
\end{equation}

We clain there is a continuous sweepout 
\begin{align*}
[-1,1] &\to \mathcal{Z}_{1}(\mathrm{Y};\mathbb{Z}_{2})\\
t &\to [[\sigma_t]],
\end{align*}
where $[[\sigma_t]]$ is the element of $\mathcal{Z}_{1}(\mathrm{Y};\mathbb{Z}_{2})$ associated to the curve $\sigma_t$, such that: 
\begin{enumerate}
\item $\sigma_{-1} = \gamma_{2}\cup\gamma_{3}$, $\sigma_1 = \gamma_{1}$.
\item $\sigma_{0} = \sigma$ and
$\max_{t \in [-1,1]} \length(\sigma_t) = \length(\sigma).$
\item For $t > 0$, $\sigma_t$ consists of a simple closed curve, while for $t < 0$, $\sigma_t$ consists of the union of two disjoint simple closed curves.
\end{enumerate}

\begin{figure}[!htb]
\centering
\includegraphics[scale=0.08]{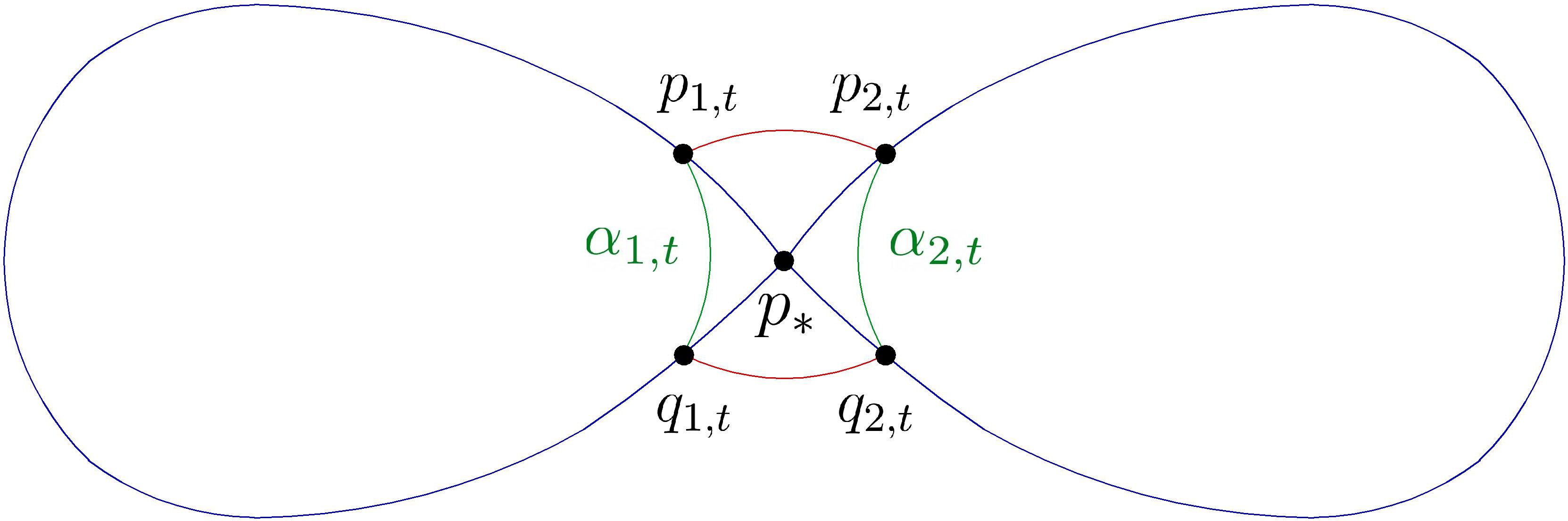}
\caption{Step 2 in the construction of the sweepout $t \to [[\sigma_t]]$.}
\label{fig4}
\end{figure}

Observe $\sigma$ divides $S$ into three connected components $\Omega_1$, $\Omega_2$ and $\Omega_3$, where $\gamma_i \subset \Omega_i$, $i=1,2,3$.
Let $p_{\ast}$ be the self-intersection point of $\sigma$. Consider $\delta < \inj(S)$. Then for every $0 < t \leq \delta$, the set $\partial B_{t}(p_{\ast})\cap\sigma$ consists of four points $p_{1,t}$, $p_{2,t}$, $q_{1,t}$, $q_{2,t}$, labelled as in figure \ref{fig4}, such that 
$$\mathrm{dist}_{S}(p_{j,t},p_{\ast}) = t,\quad \mathrm{dist}_{S}(q_{j,t},p_{\ast}) = t,$$
for all $t \in (0,\delta]$ and $j=1,2$. 
\begin{itemize}
\item[Step 1 -] We set $\sigma_{0} = \sigma$.
\item[Step 2 -] For $-\delta \leq t < 0$ and $j=1,2$, consider the geodesic segment $\alpha_{j,t}$ in $\Omega_j$ joining $p_{j,t}$ and $q_{j,t}$. Observe $\big(\sigma\backslash B_{t}(p_{\ast})\big)\cup\alpha_{1,t}\cup\alpha_{2,t}$ has length strictly less that $\sigma$, however it is not smooth. So we round the corners around $p_{j,t}$ and $q_{j,t}$ to get a smooth embedded curve $\widetilde{\sigma}_t$, which coincides with $\sigma$ outside a compact set, has two connected components (one inside $\Omega_2$ and another one inside $\Omega_3$) and satisfies
$$\length\widetilde{\sigma}_t) < \length(\sigma).$$
Then, we deform $\widetilde{\sigma}_t$ to an embedded curve $\sigma_t$ which has non zero geodesic curvature in every point and such that
$$\length(\sigma_t) < \length(\sigma).$$
Observe we can do this deformation in a such a way the path $t \mapsto \sigma_t$ is continuous in the $C^2$-topology.
\item[Step 3 -] For $0 < t \leq \delta$ we proceed analogously to the previous item, but now we join $p_{1,t}$ and $p_{2,t}$, and then $q_{1,t}$ and $q_{2,t}$, by geodesic segments contained in $\Omega_1$. Joining these curves with $\sigma\backslash B_{t}(p_{\ast})$ and rounding the corners as before, we obtain a smooth embedded connected curve $\sigma_t \subset \Omega_1$, which satisfies
$$\length(\sigma_t) < \length(\sigma).$$
\item[Step 4 -] Now, we run the \emph{curve shortening flow} with initial conditon $\sigma_{\delta}$. Then, by the work of Grayson \cite{Gr}, the solution $\sigma_t$ exists for all $t \in [\delta,+\infty)$ and converges to $\gamma_1$, as $t \to +\infty$. Moreover, the curves $\sigma_t$ are all embedded and pairwise disjoint, and determine a foliation of the domain of $\mathrm{Y}$ bounded by $\sigma_{\delta}$ and $\gamma_1$.
\item[Step 5 -] We apply the same procedure for each component $\sigma_{j,-\delta}$ of $\sigma_{-\delta}$, $j=2,3$, to obtain a foliation $\{\sigma_{j,t}\}$ of the domain bounded by $\sigma_{j,-\delta}$ and $\gamma_j$. In this case, we set $\sigma_t = \sigma_{2,t}\cup\sigma_{3,t}$.
\item[Step 6 -] The family of curves obtained is indexed by $t \in \re$. Composing with a homeomorphism $(-1,1) \to \re$ which fixes $0$ we obtain the desired curves, which we still denote by $\sigma_t$. 
\end{itemize}

%\noindent{\bf Type 1}: We consider $S$ such that:\begin{itemize}\item[(a.1)] $S$ admits a pants decomposition by systoles.\item[(a.2)] Exactly $n-1$ of the curves in the previous decompostion are separating, where $n$ is the genus of $S$.\end{itemize}

\subsection{Sharpness of the width lower bound}

Fix $a > 0$. Consider the pair of pants $\mathrm{Y}_a$ such that its three boundary geodesics have length equal to $a$. Denote these geodesics by $\gamma_{a,1},\,\gamma_{a,2},\,\gamma_{a,3}$, and let $\sigma_a$ be the figure eight geodesic of $\mathrm{Y}_a$ as in figure \ref{fig5}.

\begin{figure}[!htb]
\centering
\includegraphics[scale=0.07]{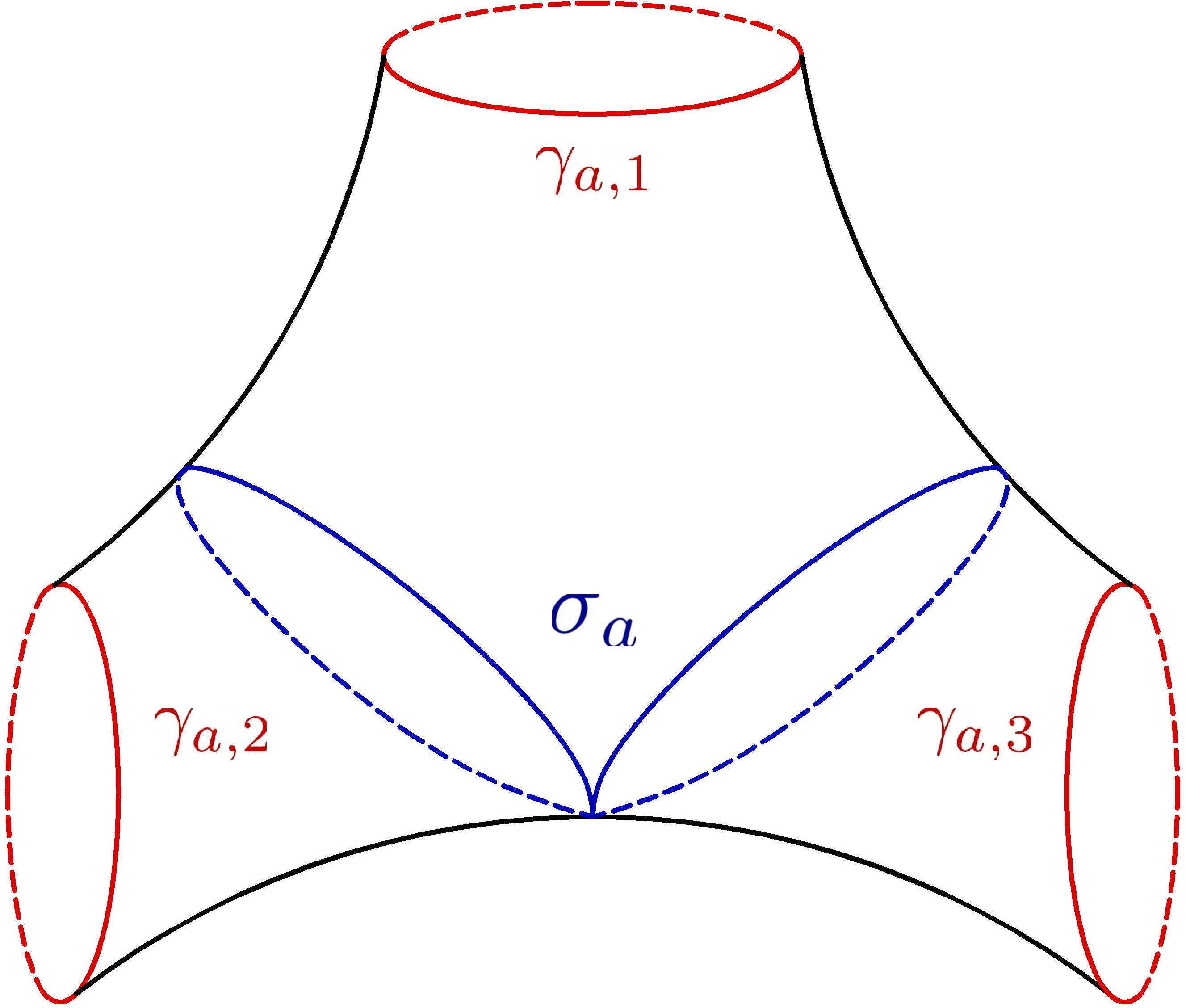}
\caption{The pair of pants $\mathrm{Y}_a$.}
\label{fig5}
\end{figure} 

Let $S_{m,a}$ be a closed hyperbolic surface of genus $m \geq 2$ which is obtained gluing copies $\mathrm{Y}_1,\ldots,\mathrm{Y}_{2g-2}$ of $\mathrm{Y}_a$ following the pattern of the figure \ref{fig6}. Denote by $\widetilde{\ga}_{a,1},\ldots,\widetilde{\ga}_{a,3m-3}$ the copies of the $\ga_{a,i}$ in $S_{g,a}$.

\begin{figure}[!htb]
\centering
\includegraphics[scale=0.06]{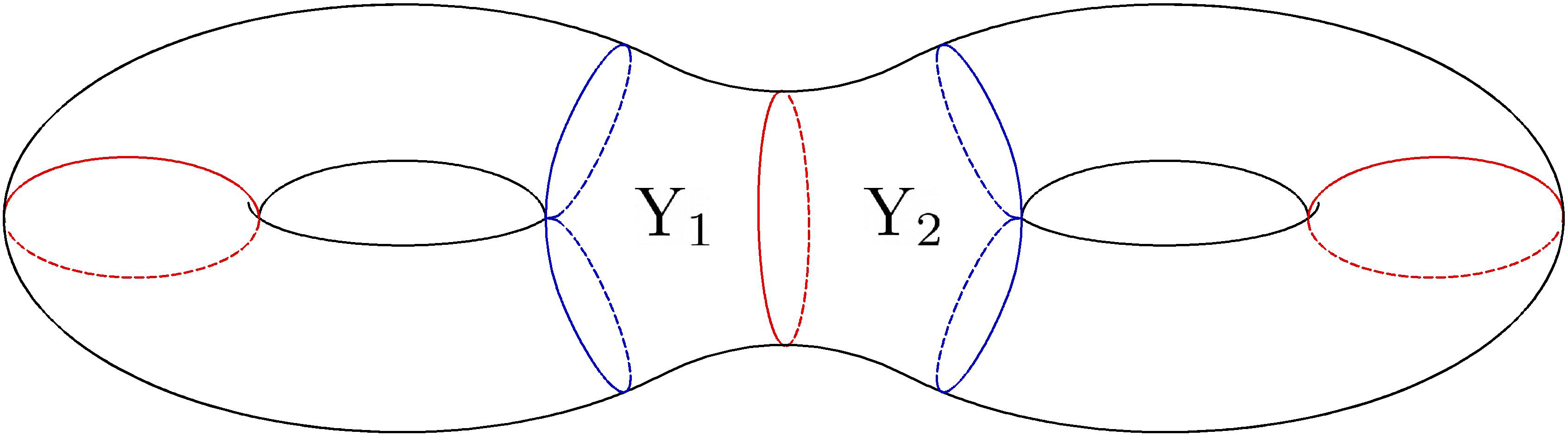}\quad
\includegraphics[scale=0.09]{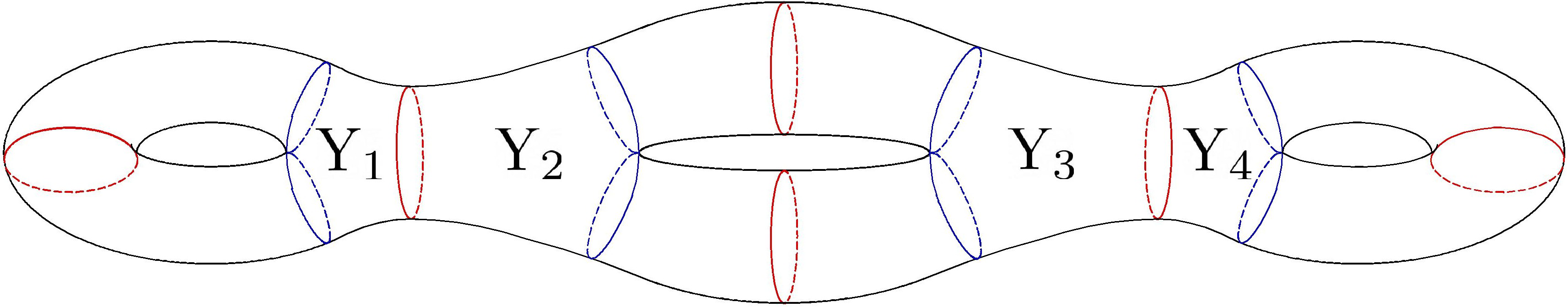}
\caption{The surfaces $S_{2,a}$ and $S_{3,a}$.}
\label{fig6}
\end{figure}

\begin{prop}\label{prop:bound_width}
Suppose the curves $\widetilde{\ga}_{a,1},\ldots,\widetilde{\ga}_{a,3m-3}$ in $S_{m,a}$ are all sytoles. Then
$$\omega_1(S_{m,a}) = \length(\sigma_a).$$
\end{prop}

\begin{proof}
On each copy of $\mathrm{Y}_a$, consider the sweepout described in subsection \ref{sec:pants}. Following the gluing pattern we obtain a sweepout $\Phi$ of $S_{m,a}$ (if we glue $\gamma_{a,1}$ and $\gamma_{a,2}$, then in $\mathcal{Z}_1(S;\mathbb{Z}_2)$ this corresponds to the zero cycle). Moreover it follows from the properties of the sweepout of the pairs of pants that $\mathbf{L}(\Phi) = \length(\sigma_a)$. Hence
$$\omega_1(S) \leq \length(\sigma_a).$$

Since $\widetilde{\ga}_{a,1},\ldots,\widetilde{\ga}_{a,3m-3}$ are sytoles, by equation \eqref{eq:length.fig.8} we conclude that a figure eight geodesic $\sigma_{\infty}$ as in the statement of Theorem \ref{thm:width.charac} necessarily satisfy $\length(\sigma_{\infty}) \geq \length(\sigma_a)$. Therefore, $\omega_1(S_{m,a}) = \length(\sigma_a)$. 
\end{proof}

\begin{thm}\label{thm:width.sharp}
There exists $a_0 > 0$ such that for all $a \in (0,a_0]$ and every integer $m \geq 2$, the width $\omega_1(S_{m,a})$ is realized by a primitive figure eight geodesic with multiplicity one and
$$\lim_{a \to 0^{+}} \omega_1(S_{m,a}) = 2\, \mathrm{arccosh}\,3.$$
\end{thm}

\begin{proof}
We claim that for $a$ small enough the curves $\widetilde{\ga}_{a,1},\ldots,\widetilde{\ga}_{a,3m-3}$ are systoles of $S_{m,a}$.

Let $\beta$ be a simple closed geodesic of $S_{m,a}$ different from $\widetilde{\ga}_{a,i}$. Then $\beta$ intersects transversely $\widetilde{\ga}_{a,j}$ for some $j$. Hence a portion $\widetilde{\beta}$ of $\beta$ is contained in the collar neighbourhood $C_{a,j}$ of $\widetilde{\ga}_{a,j}$. Observe that $\length(\widetilde\beta)$ is at least the radius of $C_{a,j}$. Since this radius blows up as $a \to 0$, for $a$ small enough we have $\length(\beta) > \length(\widetilde{\ga}_{a,j})$. Hence $\widetilde{\ga}_{a,j}$ is a systole of $S_{m,a}$.

It follows from the previous proposition that
$\omega_1(S_{m,a}) = \length(\sigma_a)$.
Moreover, by \eqref{eq:length.fig.8} we have 
$$\lim_{a \to 0^{+}}\length(\sigma_a) = 2\, \mathrm{arccosh}\,3.$$
This concludes the proof of the Theorem.
\end{proof}

\subsection{Existence non-connected limit interfaces}

Let us now prove Theorem \ref{thm:non_con}.\\

\begin{dem3}
Consider a hyperbolic surface $S_{L}$ as in figure \ref{fig7}, such that the curves $\gamma_1$, $\gamma_2$ and $\gamma_3$ have length $L$. Observe these curves divide $S_{L}$ into two isometric pair of pants. Denote by $\mathrm{Y}_{L}$ any of those pair of pants.

\begin{figure}[!htb]
\centering
\includegraphics[scale=0.08]{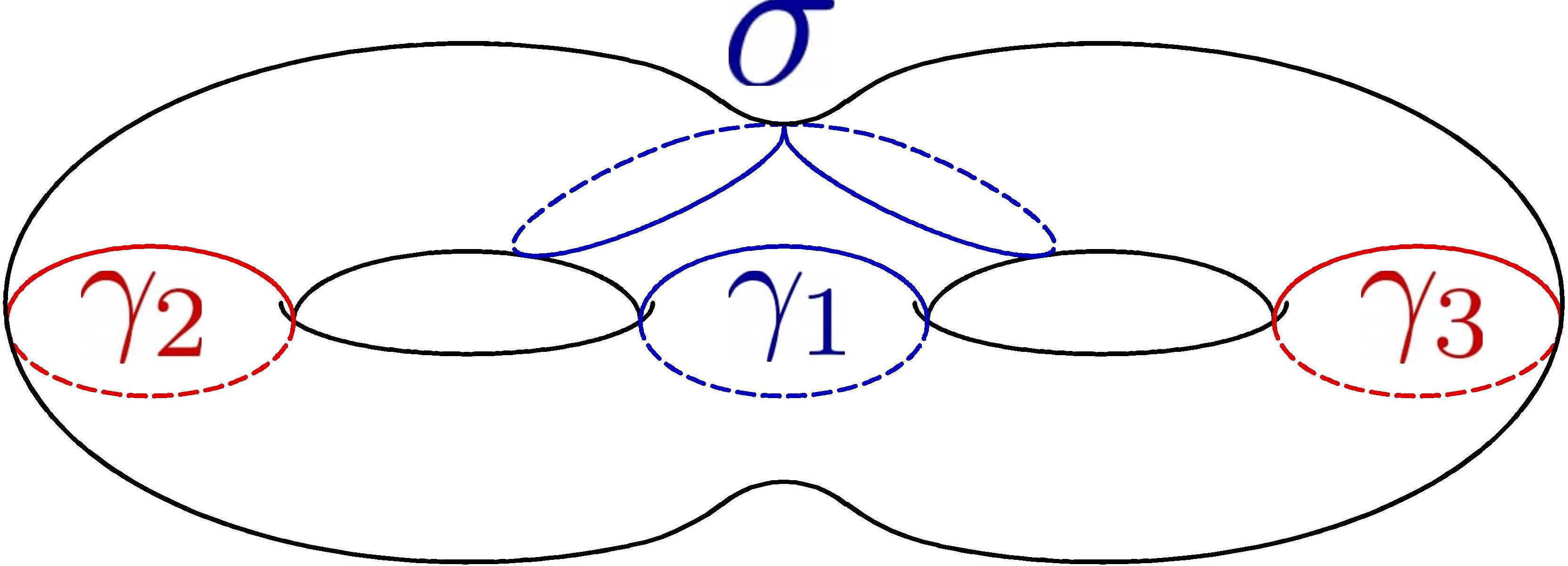}
\label{fig7}
\end{figure}

Let $\sigma_{L}$ be the figure eight depicted in figure \ref{fig7}. Consider the sweepout of $\mathrm{Y}_{L}$ described on subsection \ref{sec:pants} which starts on $[[\gamma_1]]$, ends on $[[\gamma_2\cup\gamma_3]]$ and has maximum equal to $\length(\sigma_{L})$. Joining the sweeouts on the two copies of $\mathrm{Y}_{L}$ and taking the sum of it with $[[\gamma_1]]$ we otain a sweeout of $S_{L}$ which starts and ends on the zero cycle, whose maximum is 
$\length(\sigma_{L}) + L.$ Therefore
\begin{equation}\label{eq:bound.width.nc}
\omega_1(S_{L}) \leq \length(\sigma_{L}) + L.
\end{equation}

We can argue as in the proof of Theorem \ref{thm:width.sharp} to conclude there is $L_0 > 0$ such that for $L \in (0,L_0]$, the curve $\gamma_i$ is a systole, $i=1,2,3$.

Now, let $\sigma_{\infty}$ be a figure eight geodesic as in the statement of \ref{thm:width.charac}. Suppose the length of $\sigma_{\infty}$ is different from $\length(\sigma_{L})$. Then $\sigma_{\infty}$ is not contained in any of the copies of $\mathrm{Y}$. In particular, $\sigma_{\infty}$ intersects at least one of the curves $\gamma_i$. Arguing again as in the proof of Theorem \ref{thm:width.sharp} one can choose $L_0$ such that any curve which intersects one of the $\gamma_i$, has length greater than $\length(\sigma) + L.$ However, this contradicts the inequality \eqref{eq:bound.width.nc}. Hence, $\length(\sigma_{\infty}) = \length(\sigma_{L})$. 

We conclude $\sigma_{\infty}$ is a figure eight contained in one of the copies of $\mathrm{Y}_{L}$. Since such a curve does not separate, there are simple geodesics $\sigma_1,\ldots,\sigma_N$, $N \geq 1$, such that
$$\omega_1(S_{L}) = \length(\sigma_{L}) + \sum_{i=i}^{N}\length(\sigma_i) \geq \length(\sigma_{L}) + L.$$
Therefore
$$\omega_1(S) = \length(\sigma_{L}) + L.$$ 
\end{dem3}

\subsection{The width of the Bolza surface}\label{subsec:Bolza}

The \emph{Bolza surface}, denoted by $\mathrm{B}$, is the closed hyperbolic surface of genus $2$ obtained identifying opposite edges of a regular octagon in $\mathbb{H}^2$ whose angles at the vertices are all equal to $\pi/4$, see figure \ref{fig8}.\\

\begin{figure}[!htb]
\centering
\includegraphics[scale=0.08]{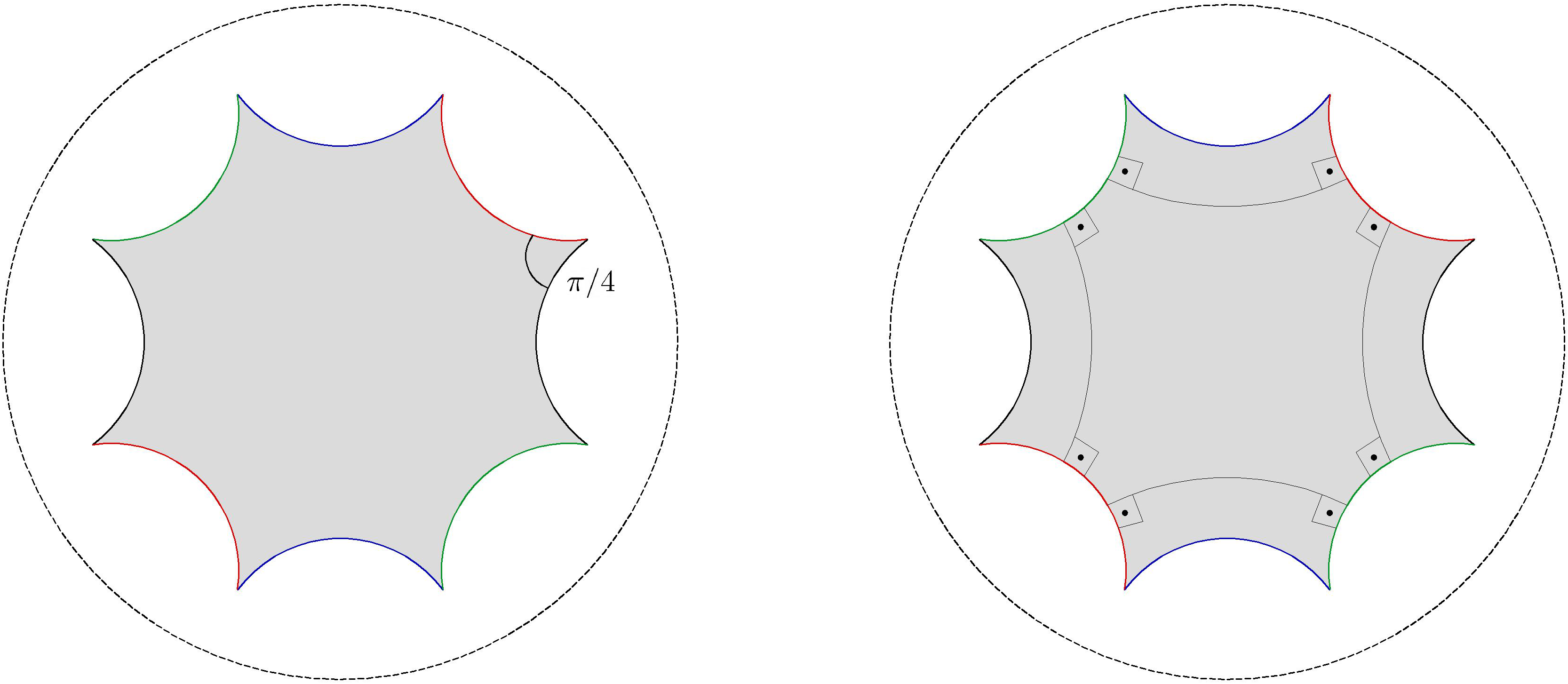}
\caption{The fundamental domain of the Bolza surface.}
\label{fig8}
\end{figure}

\begin{dem4}
Let us first obtain an upper bound for $\omega_{1}(\mathrm{B})$. We will need the following facts (see \cite[Proposition 2.6 and Theorem 5.2]{SS}): every systole of the Bolza surface is nonseparating and the value of the systole is given by
\begin{equation}\label{eq:1sys_B}
\mathrm{sys}(\mathrm{B}) = 2\,\mathrm{arcosh}(1 + \sqrt{2}) \approx 3.057.
\end{equation}
Moreover, the next value in the \emph{lenght spectrum} of $\mathrm{B}$ is
\begin{equation}\label{eq:2sys_B}
2\,\mathrm{arcosh}(3 + 2\sqrt{2}) \approx 4.897.
\end{equation}

Consider a geodesic arc $\widetilde\gamma$ in the fundamental domain of $\mathrm{B}$ which minimizes the distance between two sides separated by an edge, so that $\widetilde\gamma$ intersects the boundary of the domain perpendicularly. The union of four copies of $\widetilde\gamma$ as in figure \ref{fig8} projects to a separating closed geodesic $\gamma \subset \mathrm{B}$. 

We can compute $L_{\gamma}:=\length(\gamma)$ as follows. The arc $\widetilde\gamma$ cuts a quadrilateral from the octagon. An arc perpendicular to $\widetilde\gamma$ and passing through its midpoint divides the quarilateral into two \emph{trirectangles}. Then, hyperbolic trigonometry aplied to such a triretangle give us (see \cite[Theorem 2.3.1]{Bus})
$$L_{\gamma}:= 8\,\mathrm{arccosh}\bigg(1 + \frac{\sqrt{2}}{2}\bigg) \approx 9.027.$$

Let $\alpha$ be a separating simple closed geodesic in $\mathrm{B}$ of least length among such curves. 
Observe $\alpha$ divides $\mathrm{B}$ into two hyperbolic surfaces with geodesic boundary and signature $(1,1)$. Each such surface contains a 
interior closed geodesic whose length is less than or equal to (see \cite[Proposition 4.3]{Par})
$$2\,\mathrm{arccosh}\left(\cosh(L_{\alpha}/6) + \frac{1}{2}\right) \leq 2\,\mathrm{arccosh}\left(\cosh(L_{\gamma}/6) + \frac{1}{2}\right) \approx 3.425.$$
By \eqref{eq:1sys_B} and \eqref{eq:2sys_B} it follows such a curve is a systole of $\mathrm{B}$. So $\alpha$ and two systoles determine a pants decomposition of $\mathrm{B}$, where in each pair of pants there is a figure eigth geodesic of length $L$ satisfying
\begin{equation}\label{eq:fig.8.optimal}
L \leq 2\,\mathrm{arccosh}\big(\cosh(L_{\gamma}/2) + 2\cosh^2(\mathrm{sys}(\mathrm{B})/2)\big) \approx 9.482.
\end{equation}
As in the proof of Theorem \ref{prop:bound_width} one can then construct a sweepout whose maximum is $L$. Hence $\omega_1(\mathrm{B}) \leq L$.

Now, let $\sigma_{\infty}$ be a figure eight geodesic and $\{\sigma_i\}_{i=1}^{N}$ a collection of simple closed geodesics, as in the statement of Theorem \ref{thm:width.charac}. We will first show the collection $\{\sigma_i\}_{i=1}^{N}$ is empty.

The Bolza surface admits a decomposition into two isometric pair of pants, where each boundary geodesic is a systole \cite[page 21]{SU}.
Let $\beta$ be a figure eight geodesic contained in such a pair of pants. By \eqref{eq:length.fig.8}, its length $L_{\beta}:= \length(\beta)$, satisfies
$$\cosh(L_{\beta}/2) = \cosh(\mathrm{sys}(\mathrm{B})/2) + 2\cosh^2(\mathrm{sys}(\mathrm{B})/2)\big) = 7 + 5\sqrt{2}.$$
Thus,
\begin{equation}\label{eq:width_ub}
\mathrm{sys}(\mathrm{B}) + L_{\beta} = 2\,\mathrm{arcosh}(1 + \sqrt{2}) + 2\,\mathrm{arcosh}(7 + 5\sqrt{2}) \approx 9.729.
\end{equation}
%Observe every copy of $\beta$ is nonseparating.

Suppose the collection $\{\sigma_i\}_{i=1}^{N}$ is not empty. Then,
$$\mathrm{sys}(\mathrm{B}) + L_{\beta} \leq \omega_{1}(\mathrm{B}) \leq L,$$
which contradicts \eqref{eq:fig.8.optimal} and \eqref{eq:width_ub}. So the collection is empty. 

It follows $\sigma_{\infty}$ separates $\mathrm{B}$, so it is contained in a pair of pants where at least one the boundary components separates $\mathrm{B}$. Thus, $\length(\sigma_{\infty}) \geq L$. Therefore $\omega_{1}(\mathrm{B}) = L$.
\end{dem4}

\end{document}